\documentclass{article}

\usepackage[english]{babel}
\usepackage{amssymb,a4wide}
\usepackage{amsmath,xypic,epsf,mathrsfs,theorem}
\usepackage[all]{xy}

\newtheorem{theorem}{Theorem}[subsection]
\newtheorem{lemma}[theorem]{Lemma}
\newtheorem{corollary}[theorem]{Corollary}

\newtheorem{proposition}[theorem]{Proposition}

{\theorembodyfont{\upshape}

\newtheorem{notation}[theorem]{Notation}
\newtheorem{remark}[theorem]{Remark}
\newtheorem{remarks}[theorem]{Remarks}

\newtheorem{question}[theorem]{Question}}

\numberwithin{equation}{section}
\numberwithin{theorem}{section}

\newcommand{\spec}{\mathrm{sp}}

\newcommand{\cM}{{\cal M}}
\newcommand{\ep}{{\varepsilon}}

\newcommand{\cO}{{\cal O}}

\newcommand{\cZ}{{\cal Z}}
\newcommand{\C}{{\mathbb C}}
\newcommand{\Z}{{\mathbb Z}}
\newcommand{\N}{{\mathbb N}}
\newcommand{\D}{{\mathcal D}}

\newcommand{\R}{{\mathbb R}}
\newcommand{\Cs}{{$C^*$-al\-ge\-bra}}

\newcommand{\gange}{\! \cdot \!}

\newcommand{\setmin}{\!  \setminus}

\newcommand{\sh}{{$^*$-ho\-mo\-mor\-phism}}

\newcommand\eqdef{{\;\overset{\mbox{\scriptsize def}}{=}\;}}
\newcommand{\nprecsim}{{\operatorname{\hskip3pt \precsim\hskip-9pt |\hskip6pt}}}

\newenvironment{proof}[1][Proof:]
{\begin{trivlist}\item[]\textbf{#1} }
{\hbox{}\nobreak\hfill\quad\hbox{$\square$}\end{trivlist}}

\begin{document}

\title{The Jiang--Su algebra revisited}

\author{Mikael R\o rdam\footnote{Supported by the Danish Natural
    Science Research Council and the Fields Institute} \, and Wilhelm
  Winter\footnote{Supported by the Deutsche Forschungsgemeinschaft through the SFB 478}}


\maketitle
\begin{abstract} \noindent
We give a number of new characterizations of the Jiang--Su algebra 
$\mathcal{Z}$, both intrinsic and extrinsic, in terms of  $C^{*}$-algebraic, 
dynamical, topological and $K$-theoretic conditions. Along the way we
study divisibility properties of \Cs s, we 
give a precise characterization of those unital \Cs s of stable rank
one that admit a unital embedding of the dimension-drop \Cs{}
$Z_{n,n+1}$, and we prove a cancellation theorem for the Cuntz
semigroup of \Cs s of stable rank one. 
\end{abstract}

\section{Introduction}

In Elliott's program to classify nuclear $C^{*}$-algebras by $K$-theory data 
(see \cite{Ror:encyc} for an introduction), the systematic use of strongly 
self-absorbing $C^{*}$-algebras play a central role. The term ``strongly
self-absorbing \Cs s'' was formally coined in the paper
\cite{TomsWin:Z} to denote the class of \Cs s $\D \neq \mathbb{C}$ for which 
 there is an
isomorphism from $\D$ to $\D \otimes \D$ which is approximately
unitarily equivalent to the embedding $d \mapsto d \otimes
1$. Strongly self-absorbing \Cs s are automatically simple, nuclear
and have at most one tracial state. The Cuntz algebras $\cO_2$ and
$\cO_\infty$ and the Jiang--Su algebra $\cZ$ are strongly
self-absorbing.

Most classification results obtained so far can be interpreted as 
\emph{classification up to 
$\mathcal{D}$-stability}, where $\mathcal{D}$ is one of the (few) known 
strongly self-absorbing examples (cf.\ \cite{TomsWin:ASH}). The
classification of Kirchberg algebras can thus be viewed as
classification up to $\cO_\infty$-stability. There is at present much
interest in classification up to $\cZ$-stability, which appears to be
the largest possible class of ``$\D$-stable'' \Cs s. One may view
$\cZ$ as being the stably finite analogue of $\cO_\infty$.

The original construction of the Jiang--Su algebra in \cite{JiaSu:Z}
is as an inductive limit of a sequence of \Cs s with
specified connecting mappings. Whereas everything in this construction
in principle is concrete, the presentation is not canonical, and it
depends on infinitely many choices. Since the Jiang--Su algebra has
become to play such a central role in the classification program it is
desirable to have a more concrete and ``finite'' presentation of this
algebra, or to be able to characterize it in a more streamlined way.  
We refer to the recent paper by Dadarlat and Toms,
\cite{DadToms:Z}, for a very nice such characterization. In this paper
we present other characterizations and presentations of 
the Jiang--Su algebra. 

The many alternative 
descriptions available for the Cuntz algebra $\cO_\infty$ 
provide a guideline of what kind of characterizations 
one might expect for $\mathcal{Z}$.   They involve ($C^{*}$-)algebraic, dynamical and 
$K$-theoretic conditions; in the present paper we shall employ similar conditions to 
characterize the Jiang--Su algebra in various manners. We also give a topological 
characterization of $\mathcal{Z}$ which currently has no known analogue for 
$\mathcal{O}_{\infty}$. 

Besides its original presentation as a universal $C^{*}$-algebra with generators and 
relations,  $\mathcal{O}_{\infty}$ may be written as a crossed product of a 
canonical subalgebra by an endomorphism (see \cite{Cuntz:On}). Both these 
descriptions are concrete, and entirely intrinsic. Kirchberg has obtained a completely 
different characterization of $\mathcal{O}_{\infty}$, as the uniquely determined 
purely infinite, separable, unital, nuclear $C^{*}$-algebra which is $KK$-equivalent 
to the complex numbers (cf.\ \cite{Kir:fields}; see also
\cite{Kir:CentralSequences}).  
Note that this description is not intrinsic, since it compares
$\mathcal{O}_{\infty}$  
with the complex numbers (at least on the level of $KK$-theory). Using
the well known  
facts  that strongly self-absorbing $C^{*}$-algebras are nuclear and
either stably finite  
with a unique tracial state or 
purely infinite, it is then immediate  that $\mathcal{O}_{\infty}$ is
the unique strongly  
self-absorbing $C^{*}$-algebra that has no tracial state and is
$KK$-equivalent to $\mathbb{C}$.  
Moreover, one might rephrase the condition of being \emph{purely}
infinite in terms of the Cuntz  
semigroup: a simple $C^{*}$-algebra is purely infinite if and only if
it is infinite and has almost  
unperforated Cuntz semigroup, in which case its Cuntz semigroup
coincides with the semigroup $\{0,\infty\}$. 
Regarding the Cuntz semigroup as a
$K$-theoretic invariant  
in the broadest sense, one arrives at an abstract (but extrinsic)
characterization of  
$\mathcal{O}_{\infty}$ among strongly self-absorbing $C^{*}$-algebras
in terms of  
$K$-theory data.

Let us compare the characterizations of $\mathcal{O}_{\infty}$ and of  $\mathcal{Z}$ 
in more detail. Cuntz's original description of $\mathcal{O}_{\infty}$
uses (infinitely many)  
generators and relations. While Jiang and Su's  construction is not
quite of this type, the  
building blocks of their inductive limit are given by (finitely many)
generators and  
relations---and for many purposes this has proven to be just as useful
as  if the whole algebra  
was presented as a universal $C^{*}$-algebra. 

Cuntz's description of $\mathcal{O}_{\infty}$ as a crossed product uses the dynamics 
of a certain canonical subalgebra. It is not so easy to write the Jiang--Su algebra as a 
crossed product by a single endomorphism, since such algebras tend to have nontrivial 
$K_{1}$-groups, but we can nonetheless use dynamical properties of  certain canonical 
subalgebras to write $\mathcal{Z}$ as a stationary inductive limit of such a subalgebra; 
the connecting map is not easy to describe explicitly (its existence follows from a result of
I.~Hirshberg and the authors), but its pertinent property can 
be stated in a very elegant manner. More precisely, we show that the Jiang--Su algebra 
is a stationary inductive limit of a generalized prime dimension drop $C^{*}$-algebra 
and a trace-collapsing endomorphism; any such limit is isomorphic to $\mathcal{Z}$.    
Although the connecting maps of the inductive system are not given explicitly, this is 
still an entirely intrinsic description of the Jiang--Su algebra. We wish to point out that 
this picture has already proven highly useful in \cite{Winter:localizingEC}.

The largest part of the paper will be devoted to finite versions (for $\mathcal{Z}$) of 
Kirchberg's characterization of $\mathcal{O}_{\infty}$. The general pattern of such 
characterizations goes as follows: one states various conditions ($C^{*}$-algebraic, 
$K$-theoretic and/or topological), and shows that, if met by a
strongly self-absorbing $C^{*}$-algebra 
$\mathcal{D}$, then $\mathcal{D}$ is isomorphic to $\mathcal{Z}$. Then one 
observes that $\mathcal{Z}$ itself satisfies the conditions in question. The latter will 
follow mostly from known results by Jiang and Su, the first named author, and the second 
named author and E.~Kirchberg. To establish an isomorphism between $\mathcal{D}$ 
and $\mathcal{Z}$, it will suffice to construct unital embeddings in both directions, 
as we work within the class of strongly self-absorbing 
$C^{*}$-algebras.

Our first characterization singles out $\mathcal{Z}$ as the uniquely determined strongly 
self-absorbing $C^{*}$-algebra of stable rank one, for which the unit can be approximately 
divided in the Cuntz semigroup, and which is absorbed by any UHF algebra. The latter 
condition will guarantee that the algebra in question is absorbed by the Jiang--Su algebra, 
using a joint result by I.~Hirshberg and the authors. That the  algebra absorbs $\mathcal{Z}$ 
follows from stable rank one together with a cancellation theorem for
the Cuntz semigroup established in Section~\ref{sec:W-cancellation},  
and from the divisibility condition. The key technical tool here will be 
Proposition~\ref{prop:embedding1}, which provides criteria for embeddability of certain 
dimension drop intervals into a unital $C^{*}$-algebra. Essentially, this is done by analyzing 
a set of generators and relations quite different from those used to describe dimension drop 
intervals in \cite{JiaSu:Z}.

Along similar lines, we then obtain another characterization of $\mathcal{Z}$, as the 
uniquely determined strongly self-absorbing finite $C^{*}$-algebra which has almost 
unperforated Cuntz semigroup and which is absorbed by any UHF algbra. We point out 
that the latter condition in particular entails that the algebra in question is $KK$-equivalent 
to the complex numbers, whence this characterization indeed may be viewed as a finite 
analogue of Kirchberg's characterization of $\mathcal{O}_{\infty}$. Again, the proof 
uses ideas from Proposition~\ref{prop:embedding1} in a crucial way, along with a further 
careful analysis of divisibility properties of strongly self-absorbing $C^{*}$-algebras.

Our last characterization of the Jiang--Su algebra involves the decomposition rank, a notion 
of covering dimension for nuclear $C^{*}$-algebras introduced by E.~Kirchberg and the 
second named author in \cite{KirWinter:dr}. Our result says that $\mathcal{Z}$ is the 
uniquely determined strongly self-absorbing $C^{*}$-algebra with finite decomposition rank 
which is $KK$-equivalent to the complex numbers. The proof uses the fact that finite 
decomposition rank entails sufficient regularity on the level of the Cuntz semigroup; 
together with Proposition~\ref{prop:embedding1} this shows that finite decomposition 
rank and strongly self-absorbing imply $\mathcal{Z}$-stability. That $\mathcal{Z}$ is 
the only such algebra then follows from a recent classification theorem of the second named 
author (\cite{Winter:localizingEC}). We note that decomposition rank is of a very 
topological flavour, and that there is currently no analogous characterization for 
$\mathcal{O}_{\infty}$.

The paper is organized as follows. In Section~\ref{sec:background}, we recall some 
background results about strongly self-absorbing $C^{*}$-algebras, the Jiang--Su algebra, 
and order zero maps. In Section~\ref{sec:gen-dim-drop}, we characterize the Jiang--Su 
algebra as a stationary inductive limit of generalized dimension drop algebras. 
Section~\ref{sec:W-cancellation} provides a cancellation theorem for the Cuntz semigroup 
of $C^{*}$-algebras with stable rank one.  In Section~\ref{sec:abstract-char1} we derive 
an abstract characterization of the Jiang--Su algebra among strongly self-absorbing 
$C^{*}$-algebras of stable rank one; in the subsequent section we obtain a variation of 
this result, asking the Cuntz semigroup to be almost unperforated. Finally, in 
Section~\ref{sec:ssa-dr}, we characterize the Jiang--Su algebra among strongly 
self-absorbing $C^{*}$-algebras of finite decomposition rank.

The authors thank The Fields Institute and George Elliott for
hospitality during our stay in the fall of 2007, and we thank
George Elliott and Eberhard Kirchberg for 
a number of inspiring conversations on the question of how to characterize the Jiang--Su 
algebra abstractly.

\section{Some background results} 
\label{sec:background}

\noindent In this section we recall some well-known results about
strongly self-absorbing \Cs s in general and about the Jiang--Su
algebra, $\cZ$, in particular. (The reader is referred to the
introduction and to \cite{TomsWin:Z} for a definition and properties
of strongly self-absorbing \Cs s.) We also recall some facts about completely 
positive contractive (c.p.c.) order zero maps.

We quote below a result by Andrew Toms and the
second named author about the hierarchy of strongly self-absorbing \Cs s:

\begin{proposition}[Toms--Winter, \cite{TomsWin:Z}] \label{prop:inclusion}
Let $\mathcal{D}$ and $\mathcal{E}$ be strongly self-absorbing \Cs s. Then:
\begin{enumerate}
\item $\mathcal{D}$ embeds
unitally into $\mathcal{E}$ if and only if $\mathcal{D} \otimes \mathcal{E}$ 
is isomorphic to $\mathcal{E}$.
\item $\mathcal{D}$ and $\mathcal{E}$ are isomorphic if $\mathcal{D}$ embeds 
unitally into $\mathcal{E}$ and
  $\mathcal{E}$ embeds unitally into $\mathcal{D}$.
\end{enumerate}
\end{proposition}

\noindent For each supernatural number $p$ let $M_p$ denote the UHF
algebra of type $p$. We say that $p$ is of \emph{infinite type} if
$p^\infty = p$, in which case $M_p$ is strongly self-absorbing. (If
$p$ is a natural number, then $M_p$ will denote the \Cs{} of $p \times p$
matrices over the complex numbers.) If $p$ and $q$ are natural or
supernatural numbers, then we set
$$Z_{p,q} = \{f \in C([0,1],M_p \otimes M_q) \mid f(0) \in M_p \otimes
\C, \; f(1) \in \C \otimes M_q\}.$$
If $p$ and $q$ are natural numbers, then $Z_{p,q}$ is a so-called
\emph{dimension-drop \Cs}. If $p$ and $q$ are relatively prime, then 
$Z_{p,q}$ is said to be \emph{prime}.

It is worthwhile noting that $Z_{p,q}$ has no non-trivial projections
(other than $0$ and $1$) if and only if $p$ and $q$ are relatively
prime (natural or supernatural numbers), and that its $K$-theory in
that case is given by
$$K_0(Z_{p,q}) \cong \Z, \qquad K_1(Z_{p,q}) = 0.$$

Prime dimension-drop \Cs{}s play a crucial role in the definition of
the Jiang--Su algebra: 

\newpage

\begin{theorem}[Jiang--Su, \cite{JiaSu:Z}] \label{thm:JS}
The inductive limit of the sequence
$$A_1 \to A_2 \to A_3 \to \cdots,$$
where each $A_j$ is a prime dimension-drop \Cs{} and where the
connecting mappings are unital, is isomorphic to the Jiang--Su algebra
$\cZ$ if and only if it is simple and has a unique tracial state.
\end{theorem}

\noindent Let $p$ be a natural number. Recall from \cite{Winter:cpr1} that a 
c.p.c.\ map $\varphi \colon M_{p} \to A$ is said to have \emph{order 
  zero} if it 
preserves orthogonality. We collect below some well known facts about  
order zero maps (see \cite[Proposition~3.2(a)]{Winter:cpr1}  and 
\cite[1.2]{Winter:fintopdim} for Proposition~\ref{order-zero-facts}, and  
\cite[1.2.3]{Winter:cpr2} for Proposition~\ref{order-zero-facts-2}). We let
$e_{ij}$, or sometimes $e_{ij}^{(p)}$, denote the canonical $(i,j)$th 
matrix unit in $M_p$.

\begin{proposition}[Winter, \cite{Winter:cpr1, Winter:fintopdim}]
\label{order-zero-facts}
Let $A$ be a $C^{*}$-algebra, let $p \in \mathbb{N}$, and let
$\varphi \colon M_{p} \to A$ be a c.p.c.\ order zero map.
\begin{enumerate}
\item There is a unique $^{*}$-homomorphism $\tilde{\varphi} \colon 
\mathcal{C}_{0}((0,1]) \otimes M_{p} \to A$ such that 
$\varphi(x) = \tilde{\varphi}( \iota\otimes x)$
for all $x \in M_{p}$, where $\iota(t) = t$.
\item There is a unique $^{*}$-homomorphism $\bar{\varphi} \colon
M_{p} \to A^{**}$ 
given by sending the matrix unit $e_{ij}$ in $M_p$ to the partial
isometry in $A^{**}$ in the polar decomposition of $\varphi(e_{ij})$. We have 
\[
\varphi(x)= \bar{\varphi}(x) \varphi(1_{p}) = 
\varphi(1_{p}) \bar{\varphi}(x) 
\]
for all $x \in M_p$; and $\bar{\varphi}(1_{p})$ is the support
projection of $\varphi(1_{p})$. 
\item If, for some  $h  \in A^{**}$ with $\|h\|\le 1$, the element
  $h^{*}h$ commutes  
with $\bar{\varphi}(M_{p})$ and satisfies 
$h^{*}h \bar{\varphi}(M_{p}) \subseteq A$, then the map
$\varphi_{h} \colon M_{p} \to A$
given by $ \varphi_{h}(x) = h \bar{\varphi}(x) h^{*}$, for $x \in M_p$,
is a well defined c.p.c.\ order zero map.
\end{enumerate}
\end{proposition}

\noindent The map $\bar{\varphi}$ in (ii) above will be called the \emph{supporting 
$^{*}$-homomorphism} of $\varphi$.

\begin{proposition}[Winter, \cite{Winter:cpr2}] \label{order-zero-facts-2}
Suppose $x_1, x_{2}, \ldots,x_{p} \in A$ satisfy the relations 
\begin{equation} \tag{$\mathcal{R}_{p}$}
\|x_{i}\| \le 1, \quad x_1 \ge 0, \quad x_ix_i^* = x_1^*x_1, \quad
x_j^*x_j \perp x_i^*x_i,
\end{equation}
for all $i,j=1, \dots, n$ with $i \ne j$. Then the linear map $\psi
\colon M_p \to A$ given by $\psi(e_{ij}) = x_i^*x_j$ is a c.p.c.\
order zero map. 
\end{proposition} 

\noindent
Note that the original version of the above result was phrased in terms of elements 
of the form $e_{i1}$, $i=2,\ldots,p$. However, it is straightforward to check that the two versions 
are in fact equivalent.

The next proposition contains a recipe for finding a unital
\sh{} from a dimension drop \Cs{} $Z_{p,q}$ into a unital \Cs{} $A$. 

\begin{proposition}
\label{universal-dimension-drop}
Let $A$ be a unital  $C^{*}$-algebra. For relatively prime natural numbers 
$p$ and $q$,  suppose that $\alpha \colon M_{p} \to A$ and $\beta
\colon M_{q}\to A$ are c.p.c.\ order zero maps satisfying 
\begin{equation}
\label{alpha-beta-relations1}
\alpha(1_{p}) + \beta(1_{q}) = 1_{A}, \qquad
 [\alpha(M_{p}),\beta(M_{q})]=0.
\end{equation}
Then there is a (unique) unital $^{*}$-homomorphism $\varphi \colon
Z_{p,q} \to A$, which makes the diagram 
$$\xymatrix{& Z_{p,q} \ar[dd]_-\varphi & \\ C_0([0,1),M_p) \ar[ur]
  \ar[dr]_-{\tilde{\alpha}'} && C_0((0,1],M_q) \ar[ul]
  \ar[dl]^-{\tilde{\beta}} \\ & A &}
$$
commutative, where the upwards maps are the obvious ones, where
$\tilde{\alpha}$ and $\tilde{\beta}$ are as in
Proposition~\ref{order-zero-facts}(i), and where 
$\tilde{\alpha}'$ is obtained from $\tilde{\alpha}$ by reversing the
orientation of the interval $[0,1]$. 
\end{proposition}

\begin{proof}
By \cite[Proposition~7.3]{JiaSu:Z}, $Z_{p,q}$ is the universal $C^{*}$-algebra 
with generators $a_1, a_{2}, \ldots,a_{p}$, $b_1, b_{2},\ldots,b_{q}$
and relations ($\mathcal{R}_{p}$) from
Proposition~\ref{order-zero-facts-2} (with the $x_i$'s replaced by the
$a_i$'s), ($\mathcal{R}_{q}$) (with the $x_i$'s replaced by the
$b_i$'s), and
$$
[a_{i},b_{j}]=0, \qquad [a_{i},b_{j}^{*}]=0, \qquad \sum_{k=1}^p
a_k^*a_k + \sum_{l=1}^q b_l^*b_l = 1,
$$ 
for $i=1, \dots, p$ and $j=1, \dots, q$. Identifying $Z_{p,q}$ with a
sub-\Cs{} of $C([0,1]) \otimes M_p \otimes M_q$ in the canonical way, 
and letting $\iota
\in C([0,1])$ denote the function $\iota(t)=t$, we can take the
generators in $Z_{p,q}$ to be
$$a_i = (1-\iota)^{1/2} \otimes e_{1i}^{(p)} \otimes 1_q, \qquad b_j =
\iota^{1/2} \otimes 1_p \otimes e_{1j}^{(q)}.$$
It is straightforward to check that the elements
\[
\bar{a}_{i} = \alpha(1_{p})^{1/2} \bar{\alpha}(e_{1i}^{(p)}) =
\tilde{\alpha}'((1-\iota)^{1/2} \otimes e_{1i}^{(p)}), \qquad
\bar{b}_{j} = \beta(1_{q})^{1/2} \bar{\beta}(e_{1j}^{(q)}) =
\tilde{\beta}(\iota^{1/2} \otimes e_{1j}^{(q)})
\]
in $A$ satisfy the relations above, where $\bar{\alpha}$ and $\bar{\beta}$ 
are the supporting $^{*}$-homomorphisms for $\alpha$ and $\beta$, respectively. 
By the universal property of $Z_{p,q}$ there is (precisely) one unital
\sh{} $\varphi \colon Z_{p,q} \to A$ such that $\varphi(a_i) =
\bar{a}_i$ and $\varphi(b_j) = \bar{b}_j$ for all $i$ and $j$; and
one checks (on elements of the form $(1-\iota)^{1/2} \otimes
e_{1i}^{(p)} \in C_0([0,1),M_p)$ and $\iota^{1/2} \otimes e_{1j}^{(q)}
\in C_0((0,1],M_q)$) that the diagram in the proposition is commutative.
\end{proof}

\section{The Jiang--Su algebra and 
the \Cs s $Z_{p,q}$}
\label{sec:gen-dim-drop}

\noindent In this section we characterize the Jiang--Su algebra
using dynamical properties of the \Cs s $Z_{p,q}$ (defined in the previous section,
and with $p$ and $q$ supernatural numbers). The first result is an
immediate consequence of one of the main result from
\cite{HirRorWin:C_0(X)}:

\begin{proposition} \label{prop:ZpqD} Let $\mathcal{D}$ be a strongly
  self-absorbing \Cs{} which tensorially is absorbed by every
  UHF-algebra $B$, i.e., $\mathcal{D} \otimes B \cong B$. Then 
  $\mathcal{D} \otimes
  Z_{p,q} \cong Z_{p,q}$ whenever $p$ and $q$ are infinite supernatural
  numbers. 
\end{proposition}

\begin{proof} The \Cs{} $Z_{p,q}$ is in a canonical way 
a $C([0,1])$-algebra with fibres being UHF-algebras of type $p$ at
the left end-point, of type $pq$ at $(0,1)$, and of type $q$ at the
right end-point. Each fibre is accordingly a UHF-algebra and so absorbs 
$\mathcal{D}$ tensorially. As the interval $[0,1]$ has finite
dimension it follows from \cite{HirRorWin:C_0(X)} 
that $Z_{p,q}$ also absorbs $\mathcal{D}$.
\end{proof}

\noindent The Jiang--Su algebra is strongly self-absorbing
(\cite{TomsWin:Z}) and it is being absorbed by all UHF-algebras
(\cite{JiaSu:Z}), and so we get:

\begin{corollary} \label{cor:ZpqZ} Let $p$ and $q$ be infinite
  supernatural numbers. Then $Z_{p,q}$ absorbs the Jiang--Su algebra:
  $\cZ \otimes Z_{p,q} \cong Z_{p,q}$. 
\end{corollary}

\noindent The proposition below is proved in \cite[Proposition
2.2]{Ror:Z-absorbing} in the case where $p=n^\infty$ and $q=
m^\infty$, and where $n$ and $m$ are natural numbers, that are relatively
prime. We shall need this result in the slightly more general case where 
$p$ and $q$ are arbitrary supernatural numbers that are
relatively prime. Assume that such $p$ and $q$ are given. Then write $M_p$
and $M_q$ as inductive limits 
$$M_{p_1} \to M_{p_2} \to M_{p_3} \to \cdots \to M_p, \qquad 
M_{q_1} \to M_{q_2} \to M_{q_3} \to \cdots \to M_q$$
(with unital connecting mappings) for suitable sequences of natural
numbers $\{p_j\}$ and $\{q_j\}$. As $p_j | p$ and $q_j | q$, 
it is  automatic that 
$p_j$ and $q_j$ are relatively prime for all $j$. Let
$\sigma_j \colon M_{p_j} \otimes M_{q_j} \to M_{p_{j+1}} \otimes
M_{q_{j+1}}$ be a unital \sh{} such that $\sigma_j(M_{p_j}
\otimes \C) \subseteq M_{p_{j+1}} \otimes \C$ and $\sigma_j(\C \otimes
  M_{q_j}) \subseteq \C \otimes M_{q_{j+1}}$. Then
$Z_{p,q}$ is the limit of the inductive system
$$\xymatrix{Z_{p_1,q_1} \ar[r]^-{\rho_1} & Z_{p_2,q_2}
  \ar[r]^-{\rho_2} & Z_{p_3,q_3} \ar[r]^-{\rho_3} &  \cdots \ar[r]
  & Z_{p,q},}$$
where $\rho_j$ is given by $\rho_j(f) = \sigma_j \circ f$. Proceeding 
as in the proof of \cite[Proposition 2.2]{Ror:Z-absorbing} one obtains
the following:

\begin{proposition} \label{prop:Zpq-embed}
Let $p$ and $q$ be supernatural numbers that are relatively
prime. Then $Z_{p,q}$ embeds unitally into $\cZ$.
\end{proposition}

\noindent Combining Proposition \ref{prop:Zpq-embed} and
Corollary~\ref{cor:ZpqZ} we get unital embeddings $Z_{p,q} \to \cZ
\to Z_{p,q}$ whenever $p$ and $q$ are infinite supernatural numbers
that are relatively prime. As we shall see below, this characterizes
$\cZ$ among strongly self-absorbing \Cs s. First we note a related
result.

A unital endomorphism $\varphi$ on a unital \Cs{} $A$ is said to be
\emph{trace-collapsing} if $\tau \circ \varphi = \tau' \circ \varphi$
for any pair of tracial states $\tau$ and $\tau'$ on $A$. 

\begin{theorem} \label{thm:inductivelimit}
Let $p$ and $q$ be infinite supernatural numbers that are relatively
prime.
\begin{enumerate}
\item There exists a trace-collapsing unital endomorphism on $Z_{p,q}$.
\item Let $\varphi$ be any
trace-collapsing unital endomorphism on
$Z_{p,q}$. Then the Jiang--Su algebra $\cZ$ is isomorphic to the
inductive limit of the stationary inductive sequence:
$$\xymatrix{Z_{p,q} \ar[r]^\varphi &  Z_{p,q} \ar[r]^\varphi &
  Z_{p,q} \ar[r]^\varphi &  \cdots .}$$
\end{enumerate}
\end{theorem}

\begin{proof} (i). Take the composition of any unital embeddings $Z_{p,q} \to
\cZ \to Z_{p,q}$ (cf.\ the remarks above) and recall (eg.\ from
\cite{JiaSu:Z}) that $\cZ$ has a unique trace. 

(ii). We note first that the inductive limit, call it $A$, of
  the sequence above is an inductive limit of prime dimension-drop \Cs
  s, i.e., of \Cs s of 
  the form $Z_{n,m}$ with $n$ and $m$ natural numbers that are
  relatively prime. Indeed, each $Z_{p,q}$ is such an inductive
  limit, cf.\ the remarks above. Hence $A$ can locally be
  approximated by prime dimension-drop \Cs s. Each (prime)
  dimension-drop \Cs{} is weakly stable by 
  \cite[Proposition 7.3]{JiaSu:Z}, whence
  any \Cs{} that locally can be approximated by prime dimension-drop
  \Cs s is an
  actual inductive limit of them, cf.\ \cite{Lor:stableI}. 

It now follows from Jiang and Su, \cite{JiaSu:Z}, cf.\
Theorem~\ref{thm:JS}, that $A$ is isomorphic to the
Jiang--Su algebra $\cZ$ if and only if $A$ is simple and has unique
trace. 

Uniqueness of the trace of $A$ follows easily from the assumption that
$\varphi$ is trace-collapsing. 

The endomorphism $\varphi$ is necessarily injective. Indeed, if it
were not and $I$ is the kernel of $\varphi$, then $\varphi$ would induce an
embedding of $Z_{p,q}/I$ into $Z_{p,q}$. But any non-trivial quotient
of $Z_{p,q}$ has non-trivial projections (i.e., projections other than
$0$ and $1$), whereas $Z_{p,q}$ only contains the trivial projections, cf.\
the remarks in Section~\ref{sec:background}.

That $A$ is simple now follows from the fact that
$\varphi(a)$ is full in $Z_{p,q}$  for all non-zero $a \in
Z_{p,q}$. To see this, let $\pi_t \colon Z_{p,q} \to M_{pq}$ denote
the fibre map 
(for $t \in [0,1]$). Let $\tau$ be the (unique) tracial state on
$M_{pq}$. Then $t \mapsto (\tau \circ \pi_t \circ \varphi)(a^*a)$ is
constant by the assumption that $\varphi$ is trace-collapsing, and
this function is non-zero (because $a$ is non-zero and $\varphi$ is
injective). Hence $\pi_t(\varphi(a)) \ne 0$ for all $t \in [0,1]$, which
entails that $\varphi(a)$ is full in $Z_{p,q}$.
\end{proof}

\begin{proposition} \label{prop:embedZpq}
The Jiang--Su algebra $\cZ$ is the only strongly self-absorbing \Cs{}
for which there are relatively prime infinite supernatural numbers $p$
and $q$ and unital embeddings $Z_{p,q} \to \cZ \to Z_{p,q}$.
\end{proposition}

\begin{proof} Suppose that $p$ and $q$ are infinite supernatural
  numbers that are relatively prime and that $A$ is a strongly
  self-absorbing \Cs{} for which there are unital \sh s $\lambda
  \colon Z_{p,q} \to A$ and $\mu \colon A \to Z_{p,q}$. Consider the
  inductive system
$$\xymatrix{A \ar[r]^\mu  & Z_{p,q} \ar[r]^\lambda &  A \ar[r]^\mu  & 
Z_{p,q} \ar[r]^\lambda &  A \ar[r]^\mu  & Z_{p,q} \ar[r]^\lambda &  \cdots.}$$
The inductive limit of this system coincides with the inductive limits
of the two subsystems below:
$$\xymatrix@C+1pc{A \ar[r]^-{\lambda \circ \mu} & A \ar[r]^-{\lambda \circ
    \mu} & A \ar[r]^-{\lambda \circ \mu} & \cdots}, \qquad 
\qquad
\xymatrix@C+1pc{Z_{p,q} \ar[r]^-{\mu \circ \lambda} & Z_{p,q} 
\ar[r]^-{\mu \circ \lambda} & Z_{p,q} \ar[r]^-{\mu \circ \lambda} & \cdots.}$$

Any unital endomorphism on a strongly self-absorbing \Cs{} is
approximately unitarily equivalent to the identity by \cite[Corollary
1.12]{TomsWin:Z}. It thus follows from an inductive limit argument
(after Elliott --- see for example \cite[Corollary 2.3.3]{Ror:encyc}) 
that the former
inductive system above has inductive limit isomorphic to $A$. 

As $A$ has unique trace (cf.\  \cite[Theorem 1.7]{TomsWin:Z}) 
the unital endomorphism $\mu \circ \lambda$ is
trace-collapsing. Hence the latter of the two inductive systems above
has limit isomorphic to $\cZ$ by Theorem~\ref{thm:inductivelimit}. 

This proves that $A$ is isomorphic to $\cZ$.
\end{proof}

\section{A cancellation theorem for 
the Cuntz semigroup} 
\label{sec:W-cancellation}

\noindent In this section we prove a cancellation theorem for the
Cuntz semigroup for \Cs s of stable rank one. This result, which might
be of independent interest, and which extends a recent result of
Elliott, \cite{Ell:cancellation}, is needed for the next section. 

We refer the reader to \cite{Ror:Z-absorbing} and \cite{Ror:UHFII} for
notation and 
background material on Cuntz comparison of positive elements and on
the Cuntz semigroup. 

Recall the following fact, proved in \cite{Ror:UHFII}:

\begin{proposition} \label{prop:comparison1}
Let $A$ be a unital \Cs{} of stable rank one, let $a,b$ be
positive elements in $A$ such that $a \precsim b$, and let $\ep
>0$. It follows that there is a unitary element $u \in A$ such that
$$u^*(a-\ep)_+u \in \overline{bAb}.$$
\end{proposition}

\noindent The two results below show that the Cuntz semigroup $W(A)$
of a \Cs{} of stable rank one has almost cancellation:

\begin{proposition} \label{prop:cancellation0}
Let $A$ be a \Cs{} of stable rank one, let $a,b$ be positive elements
in $M_\infty(A)$, and let $p$ be a projection in $M_\infty(A)$ such
that
$$a \oplus p \precsim b \oplus p.$$
Then $a \precsim b$.
\end{proposition}

\begin{proof} Upon replacing $A$ by a suitable matrix algebra over $A$
  we can assume that $a,b,p$ all belong to $A$ and that $a\perp p$ and
  $b \perp p$. Let $0 < \ep < 1$. As $(p-\ep)_+ = (1-\ep)p$, we can
  use Proposition~\ref{prop:comparison1} to find a unitary $u$ in the
  unitization of $A$ such that 
$$u\big((a-\ep)_+ + p\big)u^* \in \overline{(b+p)A(b+p)} \; \eqdef \; B.$$
Being a hereditary sub-\Cs{} of $A$, $B$ and hence also its unitization are of
stable rank one. Now, $upu^*$ and $p$ are equivalent projections in $B$, and so there
is a unitary $v$ in the unitization of $B$ (that we may regard as
being a sub-\Cs{} of the unitization of $A$) such that $upu^* =
vpv^*$. Note that 
$$v^*u(a-\ep)_+u^*v \in B, \qquad v^*u(a-\ep)_+u^*v \perp v^*upu^*v = p,$$
which entails that $v^*u(a-\ep)_+u^*v$ belongs to
$(1-p)B(1-p)=\overline{bAb}$. This proves that $(a-\ep)_+ \precsim b$; and as
$\ep>0$ was arbitrary, we conclude that $a \precsim b$.
\end{proof}

\begin{theorem}[Cancellation] \label{cancellation}
Let $A$ be a \Cs{} of stable rank one, and let $x,y$ be elements in the
Cuntz semigroup $W(A)$ such that
$$x + \langle c \rangle \le y + \langle (c-\ep)_+ \rangle.
$$
for some $c \in M_\infty(A)^+$ and for some $\ep>0$. Then $x \le y$.
\end{theorem}

\begin{proof} Upon replacing $A$ by a matrix algebra over $A$ we can
  assume that $c$ belongs to $A$, and that $x = \langle a \rangle$, $y
  = \langle b \rangle$ for some positive elements $a,b$ in $A$ with $a
  \perp c$ and $b \perp c$. Next, upon adjoining a unit to $A$ we may assume that $A$
  is unital (this will not affect the comparison of the elements
  $a,b,c$). Let $h_\ep \colon \R^+ \to \R^+$ be given by
\begin{equation} \label{h}
h_\ep(t) = \begin{cases} \ep^{-1}(\ep-t), & 0 \le t \le \ep,
\\ 0, & t \ge \ep.\end{cases}
\end{equation}
Then $(c-\ep)_+ \perp h_\ep(c)$ and $c+h_\ep(c)$ is invertible. Hence
\begin{eqnarray*}
a \oplus 1_A & \precsim & a \oplus (c + h_\ep(c)) \; \precsim \; a
\oplus c \oplus h_\ep(c) \\
&\precsim & b \oplus (c-\ep)_+ \oplus h_\ep(c) \; \precsim \; b \oplus
((c-\ep)_+ + h_\ep(c)) \\
& \precsim & b \oplus 1_A.
\end{eqnarray*}
The claim now follows from Proposition~\ref{prop:cancellation0}.
\end{proof}

\noindent One cannot strengthen Theorem~\ref{cancellation} to the
more intuitive statement: $x + z \le y + z$ implies $x \le y$, when $x,y,z$ are
elements in the Cuntz semigroup, $W(A)$, of an arbitrary \Cs{} $A$ of stable
rank one. Indeed, if one takes $A$ to be a UHF algebra with trace
$\tau$, $p$ to be a projection, and $a, b$ to be positive elements in
$A$ such that  
$$ \tau(p) = d_\tau(a) \; (= \lim_{n \to \infty} \tau(a^{1/n})),$$
and such that $0$ is an accumulation point of $\spec(a)\setmin \{0\}$
and of $\spec(b)\setmin\{0\}$, then $p \nprecsim a$ but $p \oplus b
\precsim a \oplus b$ (see \cite{BPT:cuntz-semigroup} for more details).

We shall also need the lemma below for the next section. First we fix
some notation to be used here and in the sequel.

\begin{notation}
\label{g-eta-notation}
For positive numbers $0\le \eta < \ep \le 1$ define continuous 
functions $f_\ep, g_{\eta,\ep} \colon [0,1] \to \R^+$ by
\[
g_{\eta,\ep}(t)= \begin{cases}0, & t \le \eta,\\ 1, & \ep \le t \le 1,
\\ \text{linear,} & \text{else,}\end{cases} \qquad f_\ep = g_{0,\ep}. 
\]
\end{notation}

\begin{lemma} \label{lm:cancellation2}
Let $A$ be a unital \Cs{} of stable rank one, and let $a,b \in A^+$ be
such that $\langle a \rangle + \langle b  \rangle \ge \langle 1_A
\rangle$. Then $1_A - f_\ep(a) \precsim (b-\ep)_+$ for some $\ep >0$. 
\end{lemma}

\begin{proof} As $\langle (1_A - \ep)_+ \rangle = \langle 1_A \rangle$
  for all $\ep \in [0,1)$ one can conclude from \cite{Ror:UHFII} that
here exists $\delta >0$ such that 
$$\langle (a-\delta)_+ \rangle + \langle (b-\delta)_+  \rangle 
\ge \langle 1_A \rangle.$$
Take $\ep$ such that $0 < \ep < \delta$. 
Observe that $1_A - f_\ep(a) \perp (a-\ep)_+$. It follows that
\begin{eqnarray*}
\langle 1_A - f_\ep(a) \rangle + \langle (a-\ep)_+ \rangle & \le & 
\langle 1_A \rangle \; \le \;  \langle (a-\delta)_+ \rangle + \langle
(b-\delta)_+ \rangle \;   \le \; \langle (b-\ep)_+ \rangle + \langle
(a-\delta)_+ \rangle \\ & = & \langle (b-\ep)_+ \rangle + \big\langle
\big((a-\ep)_+ -(\delta-\ep)\big)_+ \big\rangle. 
\end{eqnarray*}
By Theorem~\ref{cancellation} this
implies that $1_A - f_\ep(a) \precsim (b-\ep)_+$. 
\end{proof}

\section{An axiomatic description of the Jiang--Su algebra} 
\label{sec:abstract-char1}

\noindent The main result of this section is Theorem~\ref{thm:Z-char}
below in which a new characterization of the Jiang--Su algebra is
given. The proof uses facts about the Cuntz semigroup and
comparison theory for positive elements derived in the previous
section. 

Two positive elements $a$ and $b$ in a \Cs{} $A$ are said to
be \emph{equivalent}, written $a \sim b$, 
if there is $x \in A$ such that $a=x^*x$ and $b=xx^*$. It is easy to
see that $a \sim b$ implies $\langle a \rangle = \langle b
\rangle$ in $W(A)$. Recall the definition of the dimension-drop \Cs{} $Z_{p,q}$
from Section~\ref{sec:background}.

\begin{proposition} \label{prop:embedding1}
Let $A$ be a unital \Cs{} of stable rank one, and let $n$ be a natural
number. The following four conditions are equivalent:
\begin{enumerate}
\item There exists $x \in W(A)$ such that $nx \le \langle 1_A \rangle
  \le (n+1)x$.
\item There exist $\ep > 0$ and  mutually equivalent and orthogonal 
positive elements $b_1,b_2, \dots, b_n$ in $A$ such that
$$1_A - (b_1+b_2+ \cdots + b_n) \precsim (b_1-\ep)_+.$$
\item There are elements $v,s_1, s_{2}, \ldots,s_{n} \in A$ of norm 1 such that
\begin{equation}
\label{a}
s_{1}^{*}s_{1} = s_{i}s_{i}^{*}, \qquad s_{i}^{*}s_{i} \perp
s_{j}^{*}s_{j}, \qquad v^{*}v = 1_{A} - \sum_{k=1}^{n} s_{k}^{*}s_{k},
\qquad v v^{*} s_{1}^{*}s_{1} = v v^{*},
\end{equation}
for all $i$ and $j$ with $i \ne j$. 
\item There is a unital $^{*}$-homomorphism from the \Cs{} $Z_{n,n+1}$ into $A$. 
\end{enumerate}
The hypothesis of stable rank one is only needed for the implication
{\rm (i) $\Rightarrow$ (ii)}.
\end{proposition}

\begin{proof} (i) $\Rightarrow$ (ii). Find a positive element $d$ in
  some matrix algebra $M_k(A)$ over $A$ such that $x = \langle d
  \rangle$. There is $\delta > 0$ such that $(n+1) \langle
  (d-\delta)_+ \rangle \ge \langle 1_A \rangle$ (cf.\
  \cite{Ror:UHFII}). As $n\langle d \rangle \le \langle 1_A \rangle$
there is a row matrix $t \in M_{1,nk}(A)$ such that
$$t^*t= t^*1_A t =  (d-\delta)_+ \oplus (d-\delta)_+ \oplus  \cdots \oplus
(d-\delta)_+.$$ 
Write $t = (t_1,t_2, \dots, t_n)$ with $t_i \in M_{1,k}(A)$. Then
$$t_i^*t_j = \begin{cases} (d-\delta)_+, & i=j,\\ 0, & i \ne j. \end{cases}$$
Put $e_j = t_jt_j^*$. Then $e_1, e_2, \dots, e_n$ are pairwise
orthogonal positive elements in $A$ each of which is equivalent to
$(d-\delta)_+$. It follows in particular that
$$\langle 1_A \rangle \le (n+1) \langle (d-\delta)_+ \rangle = (n+1)
\langle e_1 \rangle = 
\langle e_1 + e_2 + \cdots + e_n \rangle + \langle e_1 \rangle.$$
Now use Lemma~\ref{lm:cancellation2} (and recall from \ref{g-eta-notation} the
definition of the function $f_\eta$) to see that there exists $\eta
>0$ such that
$$1_A - f_\eta(e_1+e_2+ \cdots + e_n) \precsim (e_1-\eta)_+.$$
Put $b_j = f_\eta(e_j)$. Note that $e_1 \precsim f_\eta(e_1) = b_1$
(and also $b_1 \precsim e_1$), so there exists $\ep > 0$ such that
$(e_1-\eta)_+ \precsim (b_1-\ep)_+$ (see \cite{Ror:UHFII}). 
It now follows that the elements $b_1,b_2, \dots, b_n$ are as desired,
because $f_\eta(e_1+e_2+ \cdots + e_n) = b_1 +
b_2 + \cdots + b_n$.

(ii) $\Rightarrow$ (iii). We may assume that
$\varepsilon<1$. Since each $b_{i}$ is equivalent to $b_{1}$, there
are $x_2, \dots, x_n \in A$ such that $x_ix_i^* = b_1$ and $x_i^*x_i =
b_i$. Let $x_i = v_i|x_i| = |x_i^*|v_i$ be the polar decomposition
with $v_i$ a partial isometry in $A^{**}$. Put $s_1 = f_\ep(b_1)^{1/2}$
and put $s_i =
v_if_\ep(b_i)^{1/2} =
f_\ep(b_1)^{1/2}v_i \in A$ for $i=2,3, \dots, n$ (cf.\
\ref{g-eta-notation}). Then $s_is_i^* = f_\ep(b_1) = s_1^*s_1$ for all
$i$, and
$s_i^*s_i=f_\ep(b_i)$ whence $s_i^*s_i \perp s_j^*s_j$ when $i \ne
j$. 

Note that $1 - (f_\ep(b_1) + \cdots + f_\ep(b_n))$ belongs to the
hereditary sub-\Cs{} generated by $(1 -
(b_1+\cdots+b_n)-(1-\ep))_+$. Choose $0 < \eta < 1-\ep$ and note that
$g_{\eta,1-\ep}(1-(b_1+\cdots +b_n))$ is a unit for $(1 -
(b_1+\cdots+b_n)-(1-\ep))_+$ and hence also for $1 - (f_\ep(b_1) + \cdots +
f_\ep(b_n))$ (cf.\
\ref{g-eta-notation}). It follows from the hypothesis and
\cite[Proposition~2.4]{Ror:UHFII} that there is $x \in A$ such that 
$$x^*x = g_{\eta,1-\ep}(1-(b_1+\cdots +b_n)), \qquad xx^* \in
\overline{(b_{1}-\varepsilon)_{+} A (b_{1}-\varepsilon)_{+}}.$$
Set
\[
v= x \big(1_{A} - (f_\ep(b_{1}) + \cdots + f_\ep(b_{n}))\big)^{1/2}.
\]
Then
\[
v^{*}v = 1_{A} - (f_\ep(b_{1}) + \cdots + f_\ep(b_{n})) = 1_A -
\sum_{k=1}^n s_k^*s_k.
\]
Since $vv^*$ belongs to $\overline{(b_{1}-\varepsilon)_{+} A
  (b_{1}-\varepsilon)_{+}}$ and $(b_1-\ep)_+f_\ep(b_1)
= (b_1-\ep)_+$, we get that 
$$vv^*s_1^*s_1 = vv^*f_\ep(b_1) = vv^*.$$

(iii) $\Rightarrow$ (iv). In the light of
Proposition~\ref{universal-dimension-drop},  
it suffices to construct order zero c.p.c.\ maps 
$\alpha \colon M_{n+1} \to A$ and  $\beta \colon M_{n} \to A$ with
commuting images such that $\alpha(1_{n+1}) + \beta(1_{n}) = 1_{A}$.

The construction of $\alpha$ and $\beta$ (and the verification that
they have the desired properties) is rather long and tedious. It may
be constructive to note that one quite easily can write down order
zero c.p.c.\ maps $\mu \colon M_{n+1} \to A$ and  $\rho \colon M_{n}
\to A$ such that $\mu(1_{n+1}) + \rho(1_n) \ge 1_A$. Indeed, put
$$t_1 = (v^*v)^{1/2}, \qquad t_{j+1} = v^*s_j \quad (j=1,2, \dots,
n).$$
One can easily verify that the elements $s_1,s_2, \dots, s_n$
satisfy the relations (${\mathcal{R}}_n$) of
Proposition~\ref{order-zero-facts-2}, and that $t_1,t_2, \dots,
t_{n+1}$ satisfy the relations (${\mathcal{R}}_{n+1}$). It therefore
follows from Proposition~\ref{order-zero-facts-2} that there are order
zero c.p.c.\ maps 
$$\mu \colon M_{n+1} \to A, \quad \mu(e_{ij}^{(n+1)}) = t_i^*t_j, \qquad
\rho \colon M_n \to A, \quad \rho(e_{ij}^{(n)}) = s_i^*s_j.$$
These maps fail to have commuting images, and $\mu(1_{n+1}) +
\rho(1_n)$ is larger than but not equal to $1_A$.  We shall in
the following modify these maps so that they get the desired
properties. In the process we shall make much use of the map $\rho$, but 
we shall make no further explicit use of the map $\mu$. 
   
Upon replacing $s_1$ by $(s_1^*s_1)^{1/2}$ we may assume that
$s_{1}\ge 0$. Let $\bar{v}|v|$ be the polar decomposition of $v$ with
$\bar{v}$ a partial isometry in $A^{**}$ and $|v| = (v^*v)^{1/2}$. 
Let us note some relations satisfied by the elements $v, \bar{v}, 
s_1, \dots, s_n$ to be used later in the proof:
\begin{itemize}
\item[(a)] $s_1v=v$, $s_1 \bar{v} = \bar{v}$.
\item[(b)] $s_jv = s_j\bar{v} = 0$ for $j=2,3,\dots, n$. 
\item[(c)] $vv^* \perp v^*v$, 
\item[(d)] $s_is_j=0$ for all $i = 2,3, \dots, n$, $j=1,2,\dots,n$. 
\item[(e)] $[s_i,v^*v] = [s_i^*,v^*v] = 0$ for all $i=1, \dots, n$,
\item[(f)] $c \bar{v}^*\bar{v} = c$ for all $c \in
  \overline{v^*vAv^*v}$.
\item[(g)] $\bar{v}c \in A$ for all  $c \in \overline{v^*vAv^*v}$.
\end{itemize}

\noindent The first part of (a) follows by the hypothesis that $vv^* =
s_1^*s_1vv^*$, and the second part of (a) follows from the first part and
standard properties of the polar decomposition.
To see (b) use that $s_j^*s_jvv^* = s_j^*s_js_1^*s_1vv^*=0$. Next,
$$v^*v \gange vv^* = (1_A - \sum_{j=1}^n s_j^*s_j)vv^* 
\overset{\text{(b)}}{=} (1_A - s_1^*s_1)vv^* \overset{\text{\eqref{a}}}{=} 0,$$
whence (c) holds. For $i \ne 1$ we have $s_i^*s_is_js_j^* = s_i^*s_is_1^*s_1
=0$, so (d) holds.  For $i =1,2, \dots, n$ one has 
$$v^*vs_i = (1_A - \sum_{j=1}^n s_j^*s_j)s_i \overset{\text{(d)}}{=} 
s_i - s_1^*s_1 s_i \overset{\text{\eqref{a}}}{=} 
s_i - s_is_i^*s_i \overset{\text{(d),\eqref{a}}}{=} 
s_i(1_A - \sum_{j=1}^n s_j^*s_j)=s_iv^*v.$$
This proves (e). (f) and (g) are well-known properties of the polar
decomposition. 

Recall the definition of the order zero c.p.c.\ map $\rho$ from above,
and associate to it the
supporting \sh{} $\bar{\rho} \colon M_n \to A^{**}$ defined in
Proposition~\ref{order-zero-facts}~(ii). Note (from 
\eqref{a} and Proposition~\ref{order-zero-facts}~(ii)) that

\begin{itemize}
\item[(h)] $\rho(1_n) = 1_A-v^*v$,
\item[(i)] $\rho(1_n)\bar{\rho}(x) = \rho(x) \in A$ for all $x \in
  M_n$.
\end{itemize}
Define a map $\varphi \colon \overline{v^*vAv^*v}
\to A$ by 
$$\varphi(c) = \sum_{i=1}^n s_i^*\bar{v}c\bar{v}^*s_i, \qquad c \in
\overline{v^*vAv^*v},$$ 
cf.\ (g). The map $\varphi$ is clearly linear and hermitian, and, as
shown below, it is actually a \sh. Take $c_1,c_2 \in
\overline{v^*vAv^*v}$ and calculate:
\begin{eqnarray*}
\varphi(c_1)\varphi(c_2) &=& \sum_{i,j=1}^n
s_i^*\bar{v}c_1\bar{v}^*s_is_j^*\bar{v}c_2\bar{v}^*s_j \; = \;  
\sum_{i=1}^n s_i^*\bar{v}c_1\bar{v}^*s_is_i^*\bar{v}c_2\bar{v}^*s_i\\
&=& \sum_{i=1}^n
s_i^*\bar{v}c_1\bar{v}^*s_1^*s_1\bar{v}c_2\bar{v}^*s_i\;
\overset{\text{(a)}}{=} \; 
\sum_{i=1}^n s_i^*\bar{v}c_1\bar{v}^*\bar{v}c_2\bar{v}^*s_i\\
&\overset{\text{(f)}}{=}& \sum_{i=1}^n
s_i^*\bar{v}c_1c_2\bar{v}^*s_i \; = \; \varphi(c_1c_2),\\
\end{eqnarray*}
where we in the second and third equation have used the relations for the
$s_i$'s from \eqref{a}. 
We note the following relations concerning the \sh{} $\varphi$:
\begin{itemize}
\item[(j)] $\varphi(c)\bar{v} = \bar{v}c$ for all $c \in
  \overline{v^*vAv^*v}$,
\item[(k)] $[\varphi(c),s_i] = [\varphi(c),s_i^*]= 0$ for all $c \in
  \overline{v^*vAv^*v}$ and $i=1,2,\dots,n$, 
\item[(l)] $[\varphi(c),\bar{\rho}(x)] = 0$ for all $c \in
  \overline{v^*vAv^*v}$ and  $x \in M_n$,
\item[(m)] $\varphi(\overline{v^*vAv^*v}) \perp \overline{v^*vAv^*v}$.
\end{itemize} 
We first prove (j):  
$$\varphi(c)\bar{v} \overset{\text{(b)}}{=} 
s_1^*\bar{v}c\bar{v}^*s_1\bar{v} \overset{\text{(a)}}{=}
\bar{v}c\bar{v}^*\bar{v} \overset{\text{(f)}}{=} 
\bar{v}c.$$ 
Next, for $i=1,2, \dots, n$ we have $$\varphi(c)s_i \overset{\text{(d)}}{=}
s_1^*\bar{v}c\bar{v}^*s_1s_i  \overset{\text{(a)}}{=} \bar{v}c\bar{v}^*s_i 
\overset{\text{(a)}}{=}
s_1^*s_1\bar{v}c\bar{v}^*s_i \overset{\text{\eqref{a}}}{=} s_is_i^*\bar{v}c\bar{v}^*s_i
\overset{\text{\eqref{a}}}{=} s_i\varphi(c),$$
hence (k) holds. The image of $\bar{\rho}$ is contained in the weak closure (in $A^{**}$)
of the \Cs{} generated by the $s_i$'s, so (l) follows from (k). The
calculation:
$$\varphi(v^*v)v^*v = \sum_{j=1}^n s_i^*\bar{v}v^*v\bar{v}^*s_iv^*v = 
\sum_{j=1}^n s_i^*vv^*s_i v^*v \overset{\text{(e)}}{=} 
\sum_{j=1}^n s_i^*vv^*v^*vs_i \overset{\text{(c)}}{=} 0,$$
shows that (m) holds.

Consider the two \Cs s:
\begin{eqnarray*}
D_1 &=& \{f \in C_0([0,1),M_n) \mid f(0) \in \C \gange 1_n\},\\
D_2 &=& \{f \in C_0([0,1),M_n \otimes M_n) \mid f(0) \in \C \gange 1_n \otimes 1_{n}\}.
\end{eqnarray*}
Note (by (e)) that $v^*v$ commutes with $\rho(M_n)$ and (hence) with
$\bar{\rho}(M_n)$. 
The \sh{} $\bar{\lambda} \colon C([0,1],M_n) = C([0,1]) \otimes
M_n \to A^{**}$ given by 
$$\bar{\lambda}(f \otimes x) = f(1-v^*v) \bar{\rho}(x), \qquad f \in
C([0,1]), \; x \in M_n,$$
restricts to a \sh{} $\lambda \colon D_1 \to \overline{v^*vAv^*v}$. 
Indeed, $D_1$ is
generated as a \Cs{} 
by the elements $(1-\iota) \otimes 1_n$ and $\iota(1-\iota)
\otimes x$, $x \in M_n$, (where $\iota(t) = t$), and 
$$\bar{\lambda}((1-\iota) \otimes 1_n) = v^*v, \qquad
\bar{\lambda}(\iota(1-\iota) \otimes x) = (1-v^*v)v^*v\bar{\rho}(x)
\overset{\mathrm{(h),(i)}}{=} v^*v\rho(x) \in \overline{v^*vAv^*v}.$$
By (l) we can define a \sh{} $\gamma \colon D_2 \to A$ by
$$\gamma(f \otimes x \otimes y) = (\varphi \circ \lambda)(f \otimes
x)\bar{\rho}(y), \qquad f \in C_0([0,1)), \; x,y \in M_n. $$
To see that the image of $\gamma$ is contained in $A$ (rather than in
$A^{**}$) observe that the image of $\varphi$ is contained the
hereditary sub-\Cs{} of $A$ generated by $\sum_{i=1}^n s_i^*s_i = 1_A
- v^*v = \rho(1_n)$, and so by (i), $\varphi(c)\bar{\rho}(y)$
belongs to $A$ for all $c \in \overline{v^*vAv^*v}$ and $y \in M_n$.

Let $u \in C([0,1], M_n \otimes M_n) \cap \cM(D_2)$ be a self-adjoint
unitary such that 
$$u(t) = 1_n \otimes 1_n \quad (0 \le t \le 1/3), \qquad  
u(t)(x \otimes y)u(t) = y \otimes x \quad (2/3 \le t \le 1),$$ 
for all $x,y \in M_n$. Put $w = \gamma^{**}(u) \in A^{**}$ (where
$\gamma^{**} \colon D_2^{**} \to A^{**}$ is the canonical extension of
$\gamma$). Let $g \in C_0([0,1))$ be given by 
$$g(t) = \begin{cases} 1, & 0 \le t \le 2/3, \\
0, & t=1,\\
\text{linear}, & 2/3 < t < 1. \end{cases}$$
We list some easily verified identities involving $u$, $w$, and $g$. 
\begin{itemize}
\item[(n)] $u \gange (g \otimes 1_n \otimes 1_n) \in D_2$,
\item[(o)] $w \gange (\varphi \circ \lambda)(g \otimes 1_n) 
= (\varphi \circ \lambda)(g \otimes 1_n)\gange w
  = \gamma(u \gange (g \otimes 1_n \otimes 1_n)) \in A$.
\item[(p)] $((1-g)\otimes 1_n \otimes 1_n)\gange u \gange (1 \otimes x
  \otimes y) \gange u =
  (1-g)\otimes y \otimes x$ for all $x,y
  \in M_n$.
\end{itemize} 
Put 
$$x_i = \lambda(g \otimes 1_n)^{1/2}\bar{v}^*s_iw 
\quad (i=1,2,\dots,n), \qquad
x_{n+1} = \lambda(g\otimes 1_n)^{1/2}.$$
The $x_i$'s satisfy the following relations:
\begin{itemize}
\item[(q)] $x_i^*x_{n+1} = w(\varphi \circ \lambda)(g \otimes 1_n)
\rho(e_{i1}^{(n)})\bar{v}$ for $i=1,2, \dots, n$.
\item[(r)] $x_i^*x_j \in A$ for all
  $i,j=1,2, \dots, n+1$. 
\item[(s)] $x_jx_j^* = \lambda(g \otimes 1_n) = x_{n+1}^*x_{n+1}$ for $j=1,2,
  \dots, n+1$.
\item[(t)] $x_i^*x_i \perp x_j^*x_j$ for $i \ne j$.  
\item[(u)] $\sum_{i=1}^n s_i^*\bar{v}\bar{v}^*s_i\varphi(c) =
  \varphi(c)$ for all $c \in \overline{v^*vAv^*v}$.
\item[(v)] $\sum_{j=1}^{n+1} x_j^*x_j = \lambda(g \otimes 1_n) + 
(\varphi \circ \lambda)(g \otimes 1_n)$. 
\end{itemize} 
Let us verify these identities. For $i=1,2, \dots, n$ we have $x_i^*x_{n+1}
= ws_i^*\bar{v}\lambda(g \otimes 1_n)$ (recall that $w=w^*$) and
\begin{eqnarray*}
s_i^*\bar{v}\lambda(g \otimes 1_n) &
\overset{\text{(j)}}{=} &
s_i^*(\varphi \circ \lambda)(g \otimes 1_n) \bar{v} \; \overset{\text{(k)}}{=} \; 
(\varphi \circ \lambda)(g \otimes 1_n)s_i^* \bar{v}\\ & \overset{\text{(a)}}{=} &
(\varphi \circ \lambda)(g \otimes 1_n)s_i^*s_1 \bar{v} \; = \; 
(\varphi \circ \lambda)(g \otimes 1_n)\rho(e_{i1}^{(n)}) \bar{v}. 
\end{eqnarray*}
This proves that (q) holds. As $s_i^*\bar{v}\lambda(g \otimes 1_n)$ belongs
to $A$ (by (g)), we can use (o) to conclude that $x_i^*x_{n+1}$
belongs to $A$. When $i,j=1,2, \dots, n$ we have $x_i^*x_j  =   
w^*s_i^*\bar{v}\lambda(g \otimes 1_n)\bar{v}^*s_jw$; and
$\bar{v}\lambda(g \otimes 1_n)\bar{v}^*$ belongs to $A$ by (g). Moreover,
$$s_i^*\bar{v}\lambda(g \otimes 1_n)\bar{v}^*s_j \overset{\text{(j),(k)}}{=} 
(\varphi \circ \lambda)(g \otimes 1_n)^{1/2}
s_i^*\bar{v}\bar{v}^*s_j(\varphi \circ \lambda)(g \otimes 1_n)^{1/2}.$$
We can now use (o) to see that $x_i^*x_j$ belongs to
$A$. To see that (s) holds, note that
$\bar{v}^*s_jww^*s_j^*\bar{v} = \bar{v}^*s_1^*s_1\bar{v} =
\bar{v}^*\bar{v}$ and use (f). Use \eqref{a} to see that $x_i^*x_i \perp
x_j^*x_j$ when $i \ne j$ and $i,j \le n$. We proceed to establish (u):
\begin{eqnarray*} 
\sum_{i=1}^n s_i^*\bar{v}\bar{v}^*s_i\varphi(c) & 
\overset{\text{\eqref{a}}}{=} & 
\sum_{i=1}^n s_i^*\bar{v}\bar{v}^*s_is_i^*\bar{v}c\bar{v}^*s_i 
\; \overset{\text{\eqref{a}}}{=} \;  
\sum_{i=1}^n s_i^*\bar{v}\bar{v}^*s_1^*s_1\bar{v}c\bar{v}^*s_i 
\; \overset{\text{(a)}}{=} \; 
\sum_{i=1}^n s_i^*\bar{v}c\bar{v}^*s_i = \varphi(c).
\end{eqnarray*}
Next,
\begin{eqnarray*} \sum_{i=1}^n x_i^*x_i & = & \sum_{i=1}^n
  w^*s_i^*\bar{v}\lambda(g \otimes 1_n)\bar{v}^*s_iw \;
  \overset{\text{(j),(k)}}{=} \;  
\sum_{i=1}^n w^*(\varphi \circ \lambda)(g \otimes 1_n)
s_i^*\bar{v}\bar{v}^*s_iw  \\ & \overset{\text{(u)}}{=} &
w^*(\varphi \circ \lambda)(g \otimes 1_n)w \; \overset{\text{(o)}}{=}\; 
(\varphi \circ \lambda)(g \otimes 1_n).
\end{eqnarray*}
>From this we see that (v) holds, and we also see that $x_i^*x_i \perp
x_{n+1}^*x_{n+1}$ (cf.\ (m)). 

It follows from (r), (s), (t) and
Proposition~\ref{order-zero-facts-2} that there is an order zero
c.p.c.\ map $$\alpha \colon M_{n+1} \to A, $$
given by 
$$
\alpha(e_{ij}^{(n+1)}) = x_i^*x_j \quad (i,j=1,2, \dots, n+1).$$ 
We list some properties of $\alpha$:
\begin{itemize}
\item[(w)] $\alpha(1_{n+1}) = \lambda(g \otimes 1_n) + (\varphi \circ
  \lambda)(g \otimes 1_n)$
\item[(x)] $[\lambda(g \otimes 1_n),\bar{\rho}(M_n)] =0$, \;
  $[\alpha(1_{n+1}),\bar{\rho}(M_n)] = 0$. 
\item[(y)] $1_A - \alpha(1_{n+1}) \in \overline{\rho(1_n)A\rho(1_n)}$.
\item[(z)] $(1_A - \alpha(1_{n+1}))\bar{\rho}(M_n) \subseteq A$.
\end{itemize}
(w) is just a reformulation of (v).
The first part of (x) follows from (e) when we note that $g$ is a function of 
$1 - \iota$, whence $\lambda(g \otimes 1_n)$ belongs to the \Cs{} 
generated by $\lambda((1-\iota) \otimes 1_n) = v^*v$, and that the
image of $\bar{\rho}$ is contained in the weak closure (in $A^{**}$)
of the \Cs{} generated by the $s_i$'s.  The second part of
(x) follows from the first part together with (w) and (l). 
As $g \ge 1 -\iota$ we
get $\lambda(g \otimes 1_n) \ge \lambda((1-\iota) \otimes 1_n) = v^*v$,
whence $1_A -\alpha(1_{n+1}) \le 1_A-v^*v = \rho(1_n)$. This proves (y). Finally,
(z) follows from (y) and the fact that $\rho(1_n)\bar{\rho}(x) = \rho(x)
\in A$. 

It follows from (x) and (z) above, together with
Proposition~\ref{order-zero-facts}~(iii), that 
$$\beta(x) = (1_A-\alpha(1_{n+1}))\bar{\rho}(x), \quad x \in M_n,$$
defines an order zero c.p.c.\ map from $M_n$ into $A$. Use (y) to see that
$(1_A - \alpha(1_{n+1}))\bar{\rho}(1_n) = 1_A - \alpha(1_{n+1})$, whence
$\alpha(1_{n+1}) + \beta(1_n) = 1_A$. To complete the proof we must show
that the images of $\alpha$ and $\beta$ commute. For brevity, put $h =
\lambda(g \otimes 1_n)$, recall that $\alpha(1_{n+1}) = h +
\varphi(h)$ and that $h \perp \varphi(h)$ (the latter by (m)). 
For $k,l,i=1,2, \dots, n$ we have:
\begin{eqnarray*} 
\beta(e_{kl}^{(n)})\alpha(e_{i,n+1}^{(n+1)}) &\overset{\text{(q)}}{=}&
(1-h-\varphi(h))
\bar{\rho}(e_{kl}^{(n)})w\varphi(h)
\rho(e_{i1}^{(n)})\bar{v} \\
&\overset{\text{(i),(l),(o)}}{=}& (1-\varphi(h))\varphi(h)
\bar{\rho}(e_{kl}^{(n)})w \bar{\rho}(e_{i1}^{(n)})\rho(1_n)\bar{v}\\
&=& \gamma\Big(\big((1-g)g \otimes 1_n \otimes 1_n\big)(1 \otimes 1_n
\otimes e_{kl}^{(n)})u(1 \otimes 1_n \otimes e_{i1}^{(n)}\Big)\rho(1_n)\bar{v} \\
&\overset{\text{(p)}}{=} 
& \gamma\Big(\big((1-g)g \otimes 1_n \otimes 1_n\big)u(1 \otimes 1_n
\otimes e_{i1}^{(n)})(1 \otimes e_{kl}^{(n)} \otimes 1_n\Big)\rho(1_n)\bar{v}\\
&=& (1-\varphi(h))\varphi(h)w\bar{\rho}(e_{i1}^{(n)}) 
(\varphi \circ \lambda)(1 \otimes e_{kl}^{(n)})\rho(1_n)\bar{v} \\
&\overset{\text{(k)}}{=}& \varphi(h) 
(1-\varphi(h))w\bar{\rho}(e_{i1}^{(n)})\rho(1_n) 
(\varphi \circ \lambda)(1 \otimes e_{kl}^{(n)})\bar{v} \\
&\overset{\text{(i),(j)}}{=}& \varphi(h)(1-\varphi(h))w\rho(e_{i1}^{(n)}) \bar{v} 
\lambda(1 \otimes e_{kl}^{(n)}) \\
&\overset{\text{(o), (k), (j)}}{=}& \varphi(h)w\rho(e_{i1}^{(n)}) \bar{v}
(1-h) \bar{\rho}(e_{kl}^{(n)}) \\ 
&\overset{\text{(m)}}{=}& \varphi(h)w\rho(e_{i1}^{(n)}) \bar{v}
(1-h - \varphi(h)) \bar{\rho}(e_{kl}^{(n)}) \\ 
&=& \alpha(e_{i,n+1}^{(n+1)})\beta(e_{kl}^{(n)})
\end{eqnarray*}
The image of $\alpha$ is contained in the \Cs{}, $E$, generated by
$\{\alpha(e_{i,n+1}) \mid i=1,2, \dots n\}$, which again, by the argument
above, is contained in
the commutant of the image of $\beta$. (To see this use that 
$\alpha(1_{n+1}) \in E$, that $E$ is contained in the hereditary sub-\Cs{}
of $A$
generated by  $\alpha(1_{n+1})$, and that $\alpha(x)\alpha(y) =
\alpha(1_{n+1})\alpha(xy)$ for all $x,y \in M_{n+1}$, cf.\
Proposition~\ref{order-zero-facts}~(ii).) 
It has now been verified that the images of $\alpha$ and
$\beta$ commute. 

(iv) $\Rightarrow$ (i). This follows from \cite[Lemma
4.2]{Ror:Z-absorbing}. 
\end{proof}

\begin{remark}
\label{Z-stable-weak-comparison}
Any stably finite unital $\mathcal{Z}$-stable $C^{*}$-algebra
satisfies conditions (i) through (iv) of Proposition~\ref{prop:embedding1}. 
\end{remark}

\noindent Quite surprisingly there is a very recent example of a
unital, simple infinite dimensional \Cs{} that does not admit a unital
embedding of the 
Jiang--Su algebra or for that matter of any dimension drop $C^{*}$-algebra  
$Z_{n,m}$ with $n,m \ge 2$ (see \cite{DHTW:no-Z-embedding}). 
This example is based on
Example~4.8 of \cite{HirRorWin:C_0(X)} of a unital $C(X)$-algebra
whose fibres absorb the Jiang--Su algebra, but which does not itself
absorb the Jiang--Su algebra. This \Cs{} has no finite dimensional
quotient, and one can quite easily see that one cannot unitally embed  
the Jiang--Su algebra or $Z_{n,m}$ (with $n,m \ge 2$) into this \Cs. 

In other words, simple infinite dimensional \Cs s can fail to have
the (very weak) divisibility property \ref{prop:embedding1}~(i). Nonetheless, prompted by the
equivalence of (i) and (iv) of the 
proposition above, one might ask the following:

\begin{question} \label{q:Zembed}
Does the Jiang--Su algebra $\cZ$ embed unitally into any unital \Cs{}
$A$ for which its Cuntz semigroup $W(A)$ has the following
divisibility property: For every natural number $n$ there exists $x
\in W(A)$ such that $nx \le \langle 1_A \rangle \le (n+1)x$?
\end{question}

\noindent The Jiang--Su algebra has the divisibity property of the
question above (cf.\ \cite[Lemma 4.2]{Ror:Z-absorbing}), and hence so does any
unital \Cs{} that admits a unital embedding of $\cZ$. 
The question above has an affirmative answer when $A$ is
strongly self-absorbing and of stable rank one:

\begin{proposition} \label{prop:Zembed}
Let $\mathcal{D}$ be a strongly self-absorbing \Cs{} of stable rank one
such that for each natural
number $n$ there is $x$ in the Cuntz semigroup $W(\mathcal{D})$ with $nx \le
\langle 1_\mathcal{D} \rangle \le (n+1)x$. Then the Jiang--Su algebra $\cZ$
embeds unitally into $\mathcal{D}$.
\end{proposition}

\begin{proof} One can write $\cZ$ as an inductive limit of prime 
dimension-drop \Cs s of the form $Z_{n,n+1}$. By assumption and
  Proposition~\ref{prop:embedding1}, each $Z_{n,n+1}$ maps unitally
  into $\mathcal{D}$. As $\mathcal{D}$ is strongly self-absorbing,
  $\mathcal{D}$ embeds unitally 
  into $\mathcal{D}_\infty \cap \mathcal{D}'$, where $\mathcal{D}_\infty =
  \ell^\infty(\mathcal{D})/c_0(\mathcal{D})$, whence 
  $Z_{n,n+1}$ maps unitally into $\mathcal{D}_\infty \cap \mathcal{D}'$ for all $n$. 
  It now follows from \cite[Proposition 2.2]{TomsWin:ASH} that 
  $\mathcal{D} \cong \mathcal{D}
  \otimes \cZ$, and hence that $\cZ$ embeds unitally into $\mathcal{D}$.
\end{proof}

\noindent We are now ready to prove our main result of this section:

\begin{theorem} \label{thm:Z-char} Let $\mathcal{D}$ be a unital
  \Cs. Then $\mathcal{D} \cong \cZ$ if and only if
\begin{enumerate}
\item $\mathcal{D}$ is strongly self-absorbing,
\item the stable rank of $\mathcal{D}$ is one,
\item for all $n$ there is an element $x \in W(\mathcal{D})$ such that $nx \le
  \langle 1_\mathcal{D} \rangle \le (n+1)x$,
\item $\mathcal{D} \otimes B \cong B$ for all UHF-algebras $B$.
\end{enumerate}
\end{theorem}

\begin{proof} It is well-known that $\cZ$ satisfies (i)--(iv). To prove
  the ``if'' part it suffices to show that $\mathcal{D}$ embeds unitally into $\cZ$
  and that $\cZ$ embeds unitally into $\mathcal{D}$, cf.\
  Proposition~\ref{prop:inclusion}. 
 
It follows from Proposition~\ref{prop:ZpqD} that
$\mathcal{D} \otimes Z_{2^\infty,3^\infty} \cong Z_{2^\infty,3^\infty}$. 
Hence $\mathcal{D}$ embeds unitally into
$Z_{2^\infty,3^\infty}$ which again embeds unitally into $\cZ$ by
Proposition~\ref{prop:Zpq-embed}. 

That $\cZ$ embeds into $\mathcal{D}$ follows from
Proposition~\ref{prop:Zembed}.
\end{proof}

\section{Strongly self-absorbing \Cs s with almost unperforated Cuntz semigroup}
\label{sec:abstract-char2}

In this section, we rephrase  Theorem~\ref{thm:Z-char} in terms of an algebraic 
condition on the Cuntz semigroup of a strongly self-absorbing $C^{*}$-algebra. 
Along the way, we show that a strongly self-absorbing $C^{*}$-algebra has almost 
unperforated Cuntz semigroup if and only if it absorbs the Jiang--Su algebra.

\begin{remark}[Dimension functions] 
\label{ldf} 
A \emph{dimension function} on a \Cs{} $A$ is a function $d \colon
M_\infty(A)^+ \to \R^+$ which satisfies $d(a \oplus b) = d(a) + d(b)$,
and $d(a) \le d(b)$ if $a \precsim b$ for all $a,b \in M_\infty(A)^+$.
It is \emph{lower semicontinuous} if, for every monotone increasing
sequence $(a_n)$ in $M_\infty(A)^+$ with $a_{n} \to a$ for some $a \in A$,  
one has $d(a_n) \to d(a)$. 

If $\tau$ is a (positive) trace on $A$, then
$$d_\tau(a) = \lim_{n \to \infty} \tau(a^{1/n}) = \lim_{\ep \to 0+}
\tau(f_\ep(a)), \qquad a \in M_\infty(A)^+,$$
defines a dimension function on $A$ (where $f_\ep$ is as defined in
Notation~\ref{g-eta-notation}, and when $\tau$ is extended in the
canonical way to $M_\infty(A)$). Every lower semicontinuous
dimension function on an exact \Cs{} arises in this way. 

Every dimension function $d$ on $A$ factors through the Cuntz
semigroup, i.e., it gives rise to an additive order preserving mapping $\tilde{d} \colon
W(A) \to \R^+$ given by $\tilde{d}(\langle 
a \rangle) = d(a)$ for $a \in M_\infty(A)^+$. The functional
$\tilde{d}$ is called a state (or a dimension function) on $W(A)$. If
there is no risk of confusion, then we use the same symbol to denote
the dimension function on $A$ and the corresponding state on $W(A)$. 

It is well-known that a stably finite strongly self-absorbing $C^{*}$-algebra 
$\mathcal{D}$ has precisely one trace (which we shall usually denote by $\tau$); 
this determines a unique lower semicontinuous dimension 
function (also denoted by $d_{\tau}$ in the sequel). When identifying 
$\mathcal{D}$ with $\mathcal{D} \otimes \mathcal{D}$ one has
\begin{equation}
\label{cc}
d_{\tau}(a \otimes b) = d_{\tau}(a) \cdot d_{\tau}(b)
\end{equation}
for all $a,b \in \mathcal{D}^{+}$.
\end{remark} 
 
\begin{remark}[Almost unperforation and strict comparison]
\label{aup}
The Cuntz semigroup $W(A)$ of a \Cs{} $A$ is said to be \emph{almost
unperforated}, cf.\ \cite{Ror:Z-absorbing}, if for all $x,y \in W(A)$ and for all
natural numbers $n$ one has $(n+1)x \le ny \Rightarrow x \le y$. 

If $A$ is simple and unital, then $W(A)$ is almost unperforated if and only
if $A$ has
\emph{strict comparison}, i.e., whenever $x,y \in W(A)$ are such that $d(x) < d(y)$
for all dimension functions $d$ on $A$ (that can be taken to be
normalized: $d(\langle 1_A \rangle) = 1$), then $x \le y$ (see
\cite[Proposition~3.2]{Ror:Z-absorbing}).

If $A$ is simple, exact and unital, then $W(A)$ is almost unperforated
if and only if $A$ has \emph{strict comparison given by traces}: For all $x,y
\in W(A)$ one has that 
$x \le y$ if $d_\tau(x) < d_\tau(y)$ for all tracial
states $\tau$ on $A$, (see \cite[Corollary 4.6]{Ror:Z-absorbing}). 
\end{remark}

\begin{lemma}
\label{tau-halving}
Let $A \neq \mathbb{C}$ be a unital $C^{*}$-algebra with a faithful tracial state 
$\tau$. Then there are $0< \lambda < 1$ and positive
elements $e$ and $f$ in $A$ such that $e \perp f$ and
\[
d_{\tau}(e) = \lambda, \qquad  d_{\tau}(f) = 1-\lambda.
\]
\end{lemma}

\begin{proof}
Choose a positive normalized element $d \in A$ such that $\{0,1\}
\subseteq \sigma(d)$;  
such an element exists in any  $C^{*}$-algebra of vector space dimension strictly 
larger than 1. If $\sigma(d) \neq [0,1]$, then $A$ contains a
nontrivial projection $p$, and we can take $\lambda = \tau(p)$, $e =
p$ and $f= 1-p$. Suppose now that $\sigma(d) = [0,1]$.  
The trace $\tau$ induces a probability measure $\mu$ on $\sigma(d) = [0,1]$
which is non-zero on any non-empty open subset of $[0,1]$ (because
$\tau$ is assumed to be faithful). Take $t$ in the
open interval $(0,1)$ such that $\mu(\{t\}) = 0$. Then $\lambda =
\mu([0,t])$, $e = (d-t)_-$, and $f= (d-t)_+$ are as desired.
\end{proof}

\noindent
In the lemmas below it is established that the Cuntz
semigroup of a strongly self-absorbing \Cs{} has a rather strong
divisibility property. 

\begin{lemma} \label{halving} 
Let $\mathcal{D}$ be a
  strongly self-absorbing \Cs. There are positive elements $b,c \in \D$ 
such that $\langle b \rangle = \langle  c \rangle$, $b\perp c$, and
  $d_\tau(b)=d_\tau(c) = 1/2$.  
\end{lemma}

\begin{proof}
We can identify $\D$ with $(\D_0)^{\otimes \infty}$, where $\D_0$ is
(isomorphic to) $\D$.
By Lemma~\ref{tau-halving}, there are $0<\lambda< 1$ and  positive
elements $e,f$ in $\mathcal{D}$ (that we can assume to have norm equal
to 1) such that $e \perp f$, $d_{\tau}(e) =
\lambda$, and $d_{\tau}(f) = 1-\lambda$. Set
$\bar{\lambda}= \lambda  (1-\lambda) >0$,
and set
\[
b_{0} = e \otimes f \otimes 1_{\mathcal{D}_0} 
\otimes \cdots \in \D, \qquad c_{0}= f \otimes e \otimes 1_{\mathcal{D}_0} 
\otimes \cdots \in \D.
\]
Then $b_{0} \perp c_{0}$, and $\langle b_{0} \rangle = \langle c_{0} \rangle $ 
because $\D$ is strongly self-absorbing (which implies that there is a sequence $(u_{n})$ 
of unitaries in $\mathcal{D}$ such that $u_{n}^{*}b_{0}u_{n} \to
c_{0}$). Moreover, by \eqref{cc}, we have  
$d_{\tau}(b_{0}) = d_{\tau}(c_{0})= \bar{\lambda}$.

Set $d = e \otimes e + f \otimes f \in \D_0 \otimes \D_0$, and for
each natural number $n$ set 
$$b_n = d \otimes \cdots \otimes 
d \otimes e \otimes f \otimes 1_{\D_0} \otimes \cdots \in \D, \qquad 
c_n = d \otimes \cdots \otimes d \otimes f \otimes e \otimes 1_{\D_0}
\otimes \cdots \in D,$$
where $d$ appears $n$ times. Then, as above, we have that $\langle b_n \rangle =
\langle c_n \rangle$; and the elements $b_0,b_1,b_2, \dots$, $c_0,c_1,c_2, \dots$ are
pairwise orthogonal. Moreover, by  \eqref{cc}, we have $d_\tau(d) =
1-2\bar{\lambda}$ and hence
$$d_\tau(b_n) = d_\tau(c_n) = (1-2\bar{\lambda})^n \bar{\lambda}.$$
It follows that
$$\sum_{n=0}^\infty d_\tau(b_n) = \sum_{n=0}^\infty d_\tau(c_n) =
1/2,$$
whence the norm-convergent sums
$$b:= \sum_{n=0}^\infty \frac{1}{n+1} b_n, \qquad c:=
\sum_{n=0}^\infty \frac{1}{n+1} c_n,$$ 
define elements $b$ and $c$ in $\D$ with the desired properties.
\end{proof}

\begin{lemma} \label{approximation}
Let $A$ be a \Cs{} which contains an increasing sequence
$(A_n)$ of sub-\Cs s whose union is dense in $A$. Let $x \in
W(A)$ and $\ep>0$ be given. Let $\tau$ be a trace on $A$.
Then there exist natural numbers $k$ and $r$ and
a positive element $a \in M_r(A_k) \subseteq M_r(A)$ such that
$\langle a \rangle \le x$ and $d_\tau(\langle a \rangle) \ge
d_\tau(x)-\ep$. 
\end{lemma} 

\begin{proof} The element $x$ is represented by a positive element $b$
  in a matrix algebra $M_r(A)$ over $A$. Since $d_\tau$ is lower
  semicontinuous there is $\delta   >0$ such that
  $d_\tau((b-\delta)_+) \ge d_\tau(b) - \ep$. Find $k$ and a
  positive element $a_0$ in $M_r(A_k)$ such that $\|a_0-b\| <
  \delta/2$. Put $a = (a_0-\delta/2)_+ \in M_r(A_k)$. Then $\langle a
  \rangle \le \langle b \rangle = x$ (by \cite[Section 2]{Ror:UHFII}). Moreover, $\|a-b\| 
  < \delta$, so again by \cite[Section 2]{Ror:UHFII}, we have $\langle a \rangle \ge
  \langle (b-\delta)_+ \rangle$, which implies that $d_\tau(\langle a
  \rangle) \ge d_\tau(\langle (b-\delta)_+ \rangle) \ge d_\tau(x) -
  \ep$. 
\end{proof}

\begin{lemma} \label{splitting-into-pieces} Let $\mathcal{D}$ be a
  strongly self-absorbing \Cs. Let $x \in
  W(\mathcal{D})$ and $0 \neq k \in \N$ be given. Then, for each $\ep >0$, 
  there is $y \in W(\mathcal{D})$ such that $ky \le x$ and $kd_\tau(y) \ge
  d_\tau(x) - \ep$.  
\end{lemma}

\begin{proof}
Let us first prove the lemma for $k=2$ 
(for $k=1$, there is nothing to show). For each natural number $r$,
identify $M_{r}(\mathcal{D})$ with $M_{r}(\mathcal{D}_0) \otimes 
(\mathcal{D}_0)^{\otimes \infty}$, where $\D_0$ is (isomorphic to)
$\D$. By Lemma~\ref{approximation} it suffices to consider the case 
where $x = \langle d \rangle$ for some positive element
$$d \in M_r(\D_0) \otimes (\D_0)^{\otimes k} \otimes 1_{\D_0} \otimes
\cdots,$$
for suitable natural numbers $k$ and $r$, that is $d = d_0 \otimes 1_{\D_0} \otimes
\cdots$ for some $d_0 \in M_r(\D_0) \otimes (\D_0)^{\otimes k}$.
Let $b$ and $c$ be as in Lemma~\ref{halving},
and set
$$b' = d_0 \otimes b \in M_r(\D), \qquad c' = d_0 \otimes c \in
M_r(\D),$$
where we have identified $M_r(\D)$ with $M_r(\D_0) \otimes (\D_0)^{\otimes k}
\otimes (\D_0)^{\otimes \infty}$. Then $b'$ and $c'$ are orthogonal, 
belong to the hereditary sub-\Cs{} 
of $M_r(\D)$ generated by $d$, and satisfy 
$\langle b' \rangle = \langle c' \rangle$. Moreover, by \eqref{cc}, 
$$d_\tau(b') = d_\tau(c') = d_\tau(d)/2.$$
Set $y= \langle b' \rangle$.
Then $2y = \langle b' + c' \rangle \le
\langle d \rangle = x$, and $2d_\tau(y) = 2d_\tau(b') =
d_\tau(x)$. (Note that in this case, i.e., for $k=2$ and for $x=
\langle d \rangle$ of the special form considered above, we prove the
lemma with $\ep=0$.)

Next, a repeated application of the case $k=2$ yields that the lemma
holds for $k=2^{j}$, for any $j \in \mathbb{N}$.

To derive the lemma for an arbitrary natural number $k$, choose 
$m,j \in \mathbb{N}$ such that 
\begin{equation*}
\label{gg}
\frac{1}{k} - \frac{\varepsilon}{2k d_{\tau}(x)} \le \frac{m}{2^{j}} 
\le \frac{1}{k}.
\end{equation*}
Then 
$$2^j\big(1-\frac{\ep}{2d_\tau(x)}\big) \le mk \le 2^j.$$ 
Choose $\ep_0>0$ such that 
$$(d_\tau(x) - \ep_0)\big(1-\frac{\ep}{2d_\tau(x)}\big)   \ge d_\tau(x) - \ep.$$
Now apply the lemma with $2^{j}$ and $\ep_0$ in the place of $k$ and $\varepsilon$ 
to obtain $y_0 \in W(\D)$ with $2^j y_0 \le x$ and $2^jd_\tau(y_0) \ge
d_\tau(x) - \ep_0$. Put $y=my_0$. Then $ky = kmy_0 \le 2^jy_0 \le x$ and
$$kd_\tau(y) = mkd_\tau(y_0) \ge
2^j\big(1-\frac{\ep}{2d_\tau(x)}\big)d_\tau(y_0) \ge
\big(1-\frac{\ep}{2d_\tau(x)}\big)(d_\tau(x)-\ep_0) \ge d_\tau(x)-\ep.$$
\end{proof}

\begin{proposition}
\label{comparison-Z-stable}
Let $\mathcal{D}$ be a strongly self-absorbing \Cs. Then
$W(\mathcal{D})$ is almost unperforated if and only if $\mathcal{D}$ 
absorbs the Jiang--Su algebra tensorially.  
\end{proposition}

\begin{proof}
By \cite{Ror:Z-absorbing}, $\mathcal{Z}$-stability  implies that the
Cuntz semigroup is almost unperforated. To show the converse, it will be enough 
to consider finite $\mathcal{D}$, for if $\mathcal{D}$ is infinite, it is well 
known to absorb $\mathcal{O}_{\infty}$, hence $\mathcal{Z}$. We show
that Proposition~\ref{prop:embedding1}(ii) holds for each natural
number $n$, which then, by Proposition~\ref{prop:embedding1}, will 
imply that $Z_{n,n+1}$ embeds unitally into
$\mathcal{D}$. As in the proof of Proposition~\ref{prop:Zembed}, 
this entails that $\mathcal{Z}$ embeds unitally into $\mathcal{D}$. 
We can finally use Proposition~\ref{prop:inclusion} to conclude that  
$\mathcal{D}$ is $\mathcal{Z}$-stable.

Our proof of \ref{prop:embedding1}~(ii) follows to a large extent that of 
(i) $\Rightarrow$ (ii) of Proposition~\ref{prop:embedding1}; however, 
we will have to avoid use of Lemma~\ref{lm:cancellation2}, since we do not assume 
$\D$ to be of stable rank one.

Let $n \in \mathbb{N}$ be given. By Lemma~\ref{splitting-into-pieces} 
there is $x \in W(\mathcal{D})$ such that $nx \le \langle 1_\mathcal{D}
\rangle$ and $d_\tau(x)> 1/(n+1)$. Now follow the proof of (i)
$\Rightarrow$ (ii) of Proposition~\ref{prop:embedding1} to the point
where $\delta>0$, $d \in M_{k}(\mathcal{D})$, and pairwise orthogonal
positive elements $e_1, e_2, \dots, e_n$ in $\mathcal{D}$ have 
been constructed such that $x = \langle d \rangle$ and $e_{j} \sim (d-
\delta)_{+}$. (Note that the assumption of stable rank one was not used up to that point.)
Upon choosing $\delta >0$  small enough, and using lower
semicontinuity of $d_\tau$, one can further obtain that $d_\tau(e_1) =
d_\tau((d-\delta)_+) > 1/(n+1)$ (recalling that $d_\tau(d) = d_\tau(x)
> 1/(n+1)$). For $\eta>0$, let $f_\eta$ be as in Notation~\ref{g-eta-notation}. As $1_\mathcal{D} -
f_\eta(e_1+e_2+ \cdots +e_n) \perp (e_1+e_2+ \cdots +e_n - \eta)_+$ we
get 
\begin{eqnarray*}
\lim_{\eta \to 0+} d_\tau(1_\mathcal{D} - f_\eta(e_1+e_2+ \cdots +e_n)) & \le &
1- \lim_{\eta \to 0+} d_\tau( (e_1+e_2+ \cdots +e_n - \eta)_+) \\
&=& 1-nd_\tau(e_1)\\ &< & d_\tau(e_1) \\
& = & \lim_{\eta \to 0+}
d_\tau((e_1-\eta)_+)
\end{eqnarray*}
(the first inequality is actually equality). Thus 
we infer, by Remark~\ref{aup} and the assumption that $W(\mathcal{D})$ is
almost unperforated, that 
$$1_\mathcal{D} - f_\eta(e_1+e_2+ \cdots +e_n) \precsim (e_1-\eta)_+$$
for some $\eta > 0$. We can now follow the last three lines of the proof
of (i) $\Rightarrow$ (ii) of Proposition~\ref{prop:embedding1} to
arrive at the conclusion that \ref{prop:embedding1}(ii) holds. 
\end{proof}

\begin{corollary}
In Theorem~\ref{thm:Z-char}, conditions (ii) and (iii) may as well be replaced by
\begin{enumerate}
\item[{\rm (ii')}] $\mathcal{D}$ is  finite
\item[{\rm (iii')}] $W(\mathcal{D})$ is almost unperforated.
\end{enumerate}
\end{corollary}

\begin{proof}
By \cite{JiaSu:Z} and \cite{Ror:Z-absorbing}, $\mathcal{Z}$ satisfies
(ii') and (iii'),  
so we have to check that (ii') and (iii') (together with the other
hypotheses) imply conditions  
(ii) and (iii) of \ref{thm:Z-char}. But (iii') entails that
$\mathcal{D}$ is $\mathcal{Z}$-stable  
by Proposition~\ref{comparison-Z-stable}, and (ii') together with
$\mathcal{Z}$-stability  
yields stable rank one, cf.\ \cite{Ror:Z-absorbing}. Now by 
Remark~\ref{Z-stable-weak-comparison}, $\mathcal{D}$ satisfies
\ref{prop:embedding1}(i),  
hence \ref{thm:Z-char}(iii).
\end{proof}

\section{Strongly self-absorbing $C^{*}$-algebras with finite decomposition rank}
\label{sec:ssa-dr}

\noindent
In this final section we single out the Jiang--Su algebra among strongly 
self-absorbing $C^{*}$-algebras with finite decomposition rank. Recall 
that the latter is a notion of topological dimension for nuclear
$C^{*}$-algebras that
was introduced by E.\ Kirchberg and the second named author in 
\cite{KirWinter:dr}.

The order on the Cuntz semigroup is not the
algebraic order (i.e., if $x \le y$, then we do not necessarily have
$z$ in the Cuntz semigroup such that $y = x+z$). The following lemma,
which is needed for the proof of Proposition~\ref{dr-n-strict-comparison}
below, seeks to remedy this situation. 

\newpage
\begin{lemma} \label{lm:almostalgorder} Let $A$ be a \Cs.
\begin{enumerate}
\item Let $a,b$ be positive elements in $A$ such that $a \precsim b$,
  and let $\ep >0$ be given. Then there are positive elements $a_0$ and $c$ in
  $\overline{bAb}$ such that $$a_0 \perp c, \qquad a_0 \sim
  (a-2\ep)_+, \qquad b \precsim (a-\ep)_+ \oplus c.$$
Moreover, if $d$ is a lower semicontinuous dimension function on $A$
and if $\delta > 0$ is given, then there exists $\ep_0 > 0$ such that
if $0<\ep \le \ep_0$, then 
$$d(b)-d(a) \le d(c) \le d(b)-d(a)+\delta.$$
\item Let $d$ be a lower semicontinuous dimension function on $W(A)$,
  and let $x,y \in W(A)$ be 
  such that $x \le y$. Then, for each $\delta > 0$,
 there is $z \in W(A)$ such that $x+z \ge y$ and $d(z) \le
 d(y)-d(x)+\delta$. 
\end{enumerate}
\end{lemma}

\begin{proof} (i).
 By
  \cite[Proposition~2.4]{Ror:UHFII} there is $v \in A$ such that
  $v^*v=(a-\ep)_+$ and $vv^*$ belongs to $\overline{bAb}$. With
  $h_\ep$ as defined in \eqref{h} we have $h_\ep(vv^*) \perp
  (vv^*-\ep)_+$. (We remark that $h_\ep(vv^*)$ belongs to $A$ if $A$
  is unital, and that it otherwise belongs to the unitization of $A$.)
  Put 
$$a_0 = (vv^*-\ep)_+ \sim (v^*v-\ep)_+ = (a-2\ep)_+, \qquad 
c = h_\ep(vv^*)bh_\ep(vv^*),$$
and note that $a_0$ and $c$ both belong to $\overline{bAb}$. Moreover,
$a_0 \perp c$, and $vv^* + c$ is strictly positive in
$\overline{bAb}$. The latter implies that 
$$b \; \precsim \;  vv^* + c \; \precsim \;  
vv^* \oplus c \sim (a-\ep)_+ \oplus c.$$

If $d$ is a lower semicontinuous dimension function on $A$, then for
each $\delta >0$ there
is $\ep_0> 0$ such that $d((a-2\ep_0)_+)\ge d(a)-\delta$.
As $a_0 \perp c$ we have $d(a_0)+d(c)
= d(a_0+c) \le d(b)$, whence
$$d(c) \le d(b) - d(a_0) = d(b) - d((a-2\ep)_+) \le d(b) - d((a-2\ep_0)_+) \le
d(b)-d(a)+\delta,$$
whenever $0 < \ep \le \ep_0$.
On the other hand, since $b \precsim (a-\ep)_+ \oplus c$ we have $d(b)
\le d((a-\ep)_+) + d(c) \le d(a) +d(c)$, which entails that $d(c)\ge
d(b) -d(a)$.

(ii). Upon replacing $A$ with a matrix algebra over $A$ we can assume
that $x = \langle a \rangle$ and $y = \langle b \rangle$ for some
positive elements $a,b \in A$. Now use (i) to find $\ep>0$ and $c$
such that $b \precsim (a-\ep)_+ \oplus c \precsim a \oplus c$ and such
that $d(c) \le d(b)-d(a)+\delta$. We can then take $z$ to be $\langle
c \rangle$. 
\end{proof}

\noindent
We quote below a result by Andrew Toms and the second named author stating
that $C^{*}$-algebras with finite decomposition rank satisfy a weak
version of strict comparison. The original lemma was stated in a slightly 
different manner; the version below employs the fact that decomposition rank 
is invariant under taking matrix algebras.

\begin{lemma}[Toms--Winter, {\cite[Lemma~6.1]{TomsWinter:VI}}]
\label{dr-comparison}
Let $A$ be a simple, separable and unital $C^{*}$-algebra with
decomposition rank
$n < \infty$. Suppose that $x,y_0,y_1, \dots, y_n \in W(A)$ satisfy
$d(x) < d(y_j)$ for all $j=0,1, \dots, n$ and for any lower
semicontinuous dimension function $d$ on $A$. Then $x \le y_0+y_1 +
\cdots + y_n$. 
\end{lemma}

\noindent The lemma above has the following two sharper versions for
strongly self-absorbing \Cs s:  

\begin{lemma}
\label{ssa-n-comparison}
Let $\mathcal{D}$ be strongly self-absorbing with decomposition rank
$n<\infty$, and let $x,y \in W(\mathcal{D})$ with $(n+1) d_\tau(x) <
d_\tau(y)$ be given. Then $x \le y$.
\end{lemma}

\begin{proof}
Apply Lemma~\ref{splitting-into-pieces} with $k=n+1$ to obtain 
$z \in W(\mathcal{D})$ such that $(n+1) z \le y$ and 
$$(n+1) d_{\tau}(z) > d_{\tau}(y) - (d_{\tau}(y) - (n+1) d_{\tau}(x)) 
=  (n+1) d_{\tau}(x);$$ we then have $d_{\tau}(x) < d_{\tau}(z)$. 
Now from Lemma~\ref{dr-comparison} we obtain $x \le (n+1)z \le y $.
\end{proof} 
 
\begin{lemma}
\label{addition-domination}
Let $\mathcal{D}$ be strongly self-absorbing with decomposition rank
$n < \infty$, and let $x,y,z \in W(\mathcal{D})$ be such that $x \le y$
and $(n+1)d_\tau(z) < d_\tau(y) - d_\tau(x)$. Then $x+z \le y$.
\end{lemma}

\begin{proof}
We may assume that  $x= \langle a \rangle$, $y= \langle b \rangle$ and 
$z= \langle e \rangle$, where $a$, $b$ and $e$ are positive elements
in some matrix algebra $M_{r}(\mathcal{D})$ over $\D$. To show that
$x+z \le y$ it suffices to show that $(a-2\ep)_+ \oplus e \precsim b$
for all $\ep>0$. 

By Lemma~\ref{lm:almostalgorder}~(i) there are mutually orthogonal
positive elements 
$a_0$ and $c$ in the hereditary sub-\Cs{} of $M_r(\D)$ generated by
$b$ such that $a_0 \sim (a-2\ep)_+$ and
$d_\tau(c) \ge d_\tau(b)-d_\tau(a) >
(n+1)d_\tau(e)$. But then it follows from Lemma~\ref{ssa-n-comparison}
that $e \precsim c$, whence 
$$(a-2\ep)_+ \oplus e \; \precsim \; (a-2\ep)_+ \oplus c \; \sim
\; a_0 \oplus c \; \sim \; a_0 + c \; \precsim \;  b,$$  
as desired. 
\end{proof}

\begin{proposition}
\label{dr-n-strict-comparison}
Any strongly self-absorbing \Cs{} $\mathcal{D}$ with finite
decomposition rank absorbs the Jiang--Su algebra $\cZ$, i.e.,
$\mathcal{D} \otimes \cZ \cong \mathcal{D}$.
\end{proposition}

\begin{proof} By Remark~\ref{aup} and Proposition~\ref{comparison-Z-stable} 
it suffices to show that for all $x,y
  \in W(\mathcal{D})$ with $d_\tau(x) < d_\tau(y)$ one has $x \le
  y$, where $\tau$ is the unique trace on $\mathcal{D}$. 
Put $\delta = (d_\tau(y)-d_\tau(x))/(n+1)$, where $n$ is the
decomposition rank of $\D$.
Choose an integer $k \ge n$ 
such  that $(n+1)d_\tau(x)/k < \delta$. By
  Lemma~\ref{splitting-into-pieces} there is $x_0 \in 
  W(\mathcal{D})$ such that $kx_0 \le x$ and $kd_\tau(x_0) \ge
  d_\tau(x) - \delta/2$, and by Lemma~\ref{lm:almostalgorder}~(ii) 
there is $z \in
  W(\mathcal{D})$ such that $kx_0+z \ge x$ and $d_\tau(z) <
  d_\tau(x) - kd_\tau(x_0) + \delta/2 \le \delta$. For each
  $j=0,1,\dots,n-1$ we have
$$(n+1)d_\tau(x_0)  \; \le \; (n+1)d_\tau(x)/k \; < \; \delta \; = \; 
d_\tau(y)-d_\tau(x) \; \le \; d_\tau(y)-d_\tau(jx_0).$$
Lemma~\ref{addition-domination} therefore yields  $jx_0 \le y \Rightarrow
(j+1)x_0 \le y$ for $j=0,1,\dots,n-1$. Hence $nx_0 \le y$. Next,
$$(n+1)d_\tau(z) < (n+1)\delta \le d_\tau(y)-d_\tau(x) \le
d_\tau(y)-d_\tau(nx_0),$$
so, again by  Lemma~\ref{addition-domination}, we get $x \le nx_0+z
\le y$ as desired.
\end{proof}

\begin{theorem} 
\label{thm:Z-char-dr} 
Let $\mathcal{D}$ be a unital \Cs. Then $\mathcal{D} \cong \cZ$ if and
only if
\begin{enumerate}
\item $\mathcal{D}$ is strongly self-absorbing,
\item the decomposition rank of $\mathcal{D}$ is finite,
\item $\mathcal{D}$ is $KK$-equivalent to $\mathbb{C}$.
\end{enumerate}
\end{theorem}

\begin{proof} 
It is well-know that $\cZ$ satisfies properties (i)--(iii)
above. Assume now that (i)--(iii) holds. 
To show that $\mathcal{D} \otimes \mathcal{Z} \cong \mathcal{Z}$, note that 
$\mathcal{D} \otimes \mathcal{Z}$ and $\mathcal{Z}$ both have (locally) finite 
decomposition rank, and are $\mathcal{Z}$-stable. Since 
$\mathcal{D} \otimes \mathcal{Z}$ is  $KK$-equivalent to $\mathbb{C}$, 
$K_{0}(\mathcal{D} \otimes \mathcal{Z}) = \mathbb{Z}$ 
and $K_{1}(\mathcal{D} \otimes \mathcal{Z})=0$. Since 
$\mathcal{D} \otimes \mathcal{Z}$ is  stably finite and $\mathcal{Z}$-stable, 
the order structure of its $K$-theory is 
determined by the unique tracial state (see \cite{GongJiangSu:Z}), whence 
$\mathcal{D} \otimes \mathcal{Z} \cong \mathcal{Z}$ by 
\cite[Corollary~8.1]{Winter:localizingEC}.

That $\mathcal{D} \otimes \mathcal{Z} \cong \mathcal{D}$ simply follows  
from Proposition~\ref{dr-n-strict-comparison}. 
\end{proof}

\begin{remarks}
Formally, Theorems \ref{thm:Z-char} and \ref{thm:Z-char-dr} are very similar, 
and it is interesting to compare them. Conditions (ii) of both theorems refer to notions 
of noncommutative covering dimension; however, one should keep in mind that 
decomposition rank has a much more topological flavor than stable rank one. 

Furthermore, using \cite{HirRorWin:C_0(X)} and the fact that (generalized) 
prime dimension drop $C^{*}$-algebras are $KK$-equivalent to $\mathbb{C}$ 
(cf.\ \cite{JiaSu:Z}), \ref{thm:Z-char-dr}(iii) follows from   \ref{thm:Z-char}(iv). 
Because of these conditions, neither \ref{thm:Z-char} nor \ref{thm:Z-char-dr} 
are completely intrinsic characterizations

Condition \ref{thm:Z-char}(iii)  may be interpreted as a $K$-theory type  condition 
in the broadest sense; it remains an interesting possibility that it is redundant in 
\ref{thm:Z-char}. Similarly, it might be the case that \ref{thm:Z-char-dr} still 
holds when only asking for locally finite (as opposed to finite) decomposition rank 
in \ref{thm:Z-char-dr}(ii). Conditions \ref{thm:Z-char}(iii) and \ref{thm:Z-char-dr}(ii) 
are (implicitly) both used to ensure   notions of comparison of positive elements. 
So,  the question is whether (stably finite) strongly self-absorbing $C^{*}$-algebras 
automatically have some sort of comparison property. (In the infinite case, this has an 
affirmative answer, since an infinite strongly self-absorbing $C^{*}$-algebra is always 
purely infinite by a result of Kirchberg.)
\end{remarks}

\bibliographystyle{amsplain}

\begin{thebibliography}{10}

\bibitem{BPT:cuntz-semigroup}
N.~Brown, F.~Perera, and A.~S. Toms, \emph{{The Cuntz semigroup, the Elliott
  conjecture, and dimension functions on $C^{*}$-algebras}}, To appear in J.
  Reine Angew. Math., 2006.

\bibitem{Cuntz:On}
J.~Cuntz, \emph{{Simple $C^*$-algebras generated by isometries}}, Comm. Math.
  Phys. \textbf{57} (1977), 173--185.

\bibitem{DadToms:Z}
M.~Dadarlat and A.~Toms, \emph{{A universal property for the Jiang-Su
  algebra}}, preprint.

\bibitem{DHTW:no-Z-embedding}
M.~D\u{a}d\u{a}rlat, I.~Hirshberg, A.~S. Toms, and W.~Winter, \emph{{The
  Jiang--Su algebra does not always embed}}, Preprint, Math. Archive
  math.OA/0712.2020v1, 2007.

\bibitem{Ell:cancellation}
G.~A. Elliott, \emph{{Hilbert modules over $C^{*}$-algebras of stable rank
  one}}, C. R. Math. Acad. Sci. Soc. R. Can. \textbf{29} (2007), 48--51.

\bibitem{GongJiangSu:Z}
G.~Gong, X.~Jiang, and H.~Su, \emph{{Obstructions to $\mathcal{Z}$-stability
  for unital simple $C^*$-algebras}}, Canadian Math. Bull. \textbf{43} (2000),
  no.~4, 418--426.

\bibitem{HirRorWin:C_0(X)}
I.~Hirshberg, M.~R{\o}rdam, and W.~Winter, \emph{{$C_0(X)$-algebras, stability
  and strongly self-absorbing $C^*$-algebras}}, Math. Ann. \textbf{339} (2007),
  no.~3, 695--732.

\bibitem{JiaSu:Z}
X.~Jiang and H.~Su, \emph{{On a simple unital projectionless $C^*$-algebra}},
  American J. Math. \textbf{121} (1999), no.~2, 359--413.

\bibitem{Kir:fields}
E.~Kirchberg, \emph{{The classification of Purely Infinite $C^*$-algebras using
  Kasparov's Theory}}, in preparation.

\bibitem{Kir:CentralSequences}
\bysame, \emph{{Central sequences in $C^*$-algebras and strongly purely
  infinite $C^*$-algebras}}, Abel Symposia \textbf{1} (2006), 175--231.

\bibitem{KirWinter:dr}
E.~Kirchberg and W.~Winter, \emph{{Covering dimension and quasidiagonality}},
  Int. J. Math. \textbf{15} (2004), 63--85.

\bibitem{Lor:stableI}
T.~Loring, \emph{{$C^*$-algebras generated by stable relations}}, J. Funct.
  Anal. \textbf{112} (1993), 159--203.

\bibitem{Ror:UHFII}
M.~R{\o}rdam, \emph{{On the Structure of Simple $C^*$-algebras Tensored with a
  UHF-Algebra, II}}, J. Funct. Anal. \textbf{107} (1992), 255--269.

\bibitem{Ror:encyc}
\bysame, \emph{{Classification of Nuclear, Simple $C^*$-algebras}},
  {Classification of Nuclear $C^*$-Algebras. Entropy in Operator Algebras}
  (J.~Cuntz and V.~Jones, eds.), vol. 126, Encyclopaedia of Mathematical
  Sciences. Subseries: Operator Algebras and Non-commutative Geometry, no. VII,
  Springer Verlag, Berlin, Heidelberg, 2001, pp.~1--145.

\bibitem{Ror:Z-absorbing}
\bysame, \emph{{The stable and the real rank of $\mathcal{Z}$-absorbing
  C*-algebras}}, International J.\ Math. \textbf{15} (2004), no.~10,
  1065--1084.

\bibitem{TomsWin:ASH}
A.~Toms and W.~Winter, \emph{{$\mathcal{Z}$-stable ASH algebras}}, Preprint,
  Math. Archive math.OA/0508218, to appear in Canadian J.\ Math., 2005.

\bibitem{TomsWinter:VI}
\bysame, \emph{{The Elliott conjecture for Villadsen algebras of the first
  type}}, Preprint, Math. Archive math.OA/0611059, 2006.

\bibitem{TomsWin:Z}
\bysame, \emph{{Strongly self-absorbing C$^*$-algebras}}, Transactions AMS
  \textbf{359} (2007), 3999--4029.

\bibitem{Winter:cpr2}
W.~Winter, \emph{{Covering dimension for nuclear $C^{*}$-algebras II}},
  (2001), Preprint, Math. Archive math.OA/0108102, to appear in Transactions
  AMS.

\bibitem{Winter:cpr1}
\bysame, \emph{{Covering dimension for nuclear $C^{*}$-algebras}}, J. Funct.
  Anal. \textbf{199} (2003), 535--556.

\bibitem{Winter:fintopdim}
\bysame, \emph{{On topologically finite dimensional simple $C^{*}$-algebras}},
  Math. Ann. \textbf{332} (2005), 843--878.

\bibitem{Winter:localizingEC}
\bysame, \emph{{Localizing the Elliott conjecture at strongly self-absorbing
  $C^{*}$-algebras}}, Preprint, Math. Archive math.OA/0708.0283v3, 2007.

\end{thebibliography}
\providecommand{\bysame}{\leavevmode\hbox to3em{\hrulefill}\thinspace}
\providecommand{\MR}{\relax\ifhmode\unskip\space\fi MR }
\providecommand{\MRhref}[2]{%
  \href{http://www.ams.org/mathscinet-getitem?mr=#1}{#2}
}
\providecommand{\href}[2]{#2}

\vspace{.5cm}

\noindent{\sc Department of Mathematical Sciences, University of
  Copenhagen, Universitets\-parken 5, DK-2100 Copenhagen, Denmark}

\vspace{.3cm}

\noindent{\sl E-mail address:} {\tt rordam@math.ku.dk}\\
\noindent{\sl Internet home page:}
{\tt www.math.ku.dk/$\,\widetilde{\;}\;$rordam} \\

\vspace{.5cm}
\noindent{\sc School of Mathematical Sciences, University of Nottingham, 
University Park, Nottingham NG7 2RD, United Kingdom}

\vspace{.3cm}
\noindent{\sl E-mail address:} {\tt wilhelm.winter@nottingham.ac.uk}\\
\end{document}